\documentclass[11pt]{amsart}
\usepackage{amssymb}
\usepackage{amsmath}
\usepackage{amscd}
\usepackage{amsfonts}
\usepackage{mathrsfs}
\usepackage{stmaryrd}
\usepackage{tikz}
\usetikzlibrary{arrows}
\usepackage[T1]{fontenc}
\allowdisplaybreaks

\makeatletter
\DeclareFontFamily{U}{tipa}{}
\DeclareFontShape{U}{tipa}{m}{n}{<->tipa10}{}
\newcommand{\arc@char}{{\usefont{U}{tipa}{m}{n}\symbol{62}}}%

\newcommand{\arc}[1]{\mathpalette\arc@arc{#1}}

\newcommand{\arc@arc}[2]{%
  \sbox0{$\m@th#1#2$}%
  \vbox{
    \hbox{\resizebox{\wd0}{\height}{\arc@char}}
    \nointerlineskip
    \box0
  }%
}
\makeatother

\usepackage{bbm}

\theoremstyle{definition}
\newtheorem{theorem}{Theorem}[section]

\newtheorem{thm}[theorem]{Theorem}
\newtheorem{prop}[theorem]{Proposition}

\newtheorem{defn}[theorem]{Definition}
\newtheorem{lemma}[theorem]{Lemma}

\newtheorem{prop-def}{Proposition-Definition}[section]

\newtheorem{rema}[theorem]{Remark}

\newcommand{\R}{{\mathbb R}}
\newcommand{\N}{{\mathbb N}}
\newcommand{\C}{{\mathbb C}}
\newcommand{\Z}{{\mathbb Z}}

\newcommand{\one}{\mathbf{1}}
\renewcommand{\d}{\mathbf{d}}
\newcommand{\wt}{\mbox{\rm wt}\ }
\newcommand{\tr}{\textrm{Tr }}

\begin{document}

\setlength{\oddsidemargin}{0cm} \setlength{\evensidemargin}{0cm}
\baselineskip=18pt

\title{On modules for meromorphic open-string vertex algebras}

\author{Fei Qi}

\begin{abstract}
We study modules of the meromorphic open-string vertex algebra (MOSVAs hereafter), a noncommutative generalization of vertex (operator) algebra defined by Yi-Zhi Huang. We start by recalling the definition of a MOSVA $V$ and left $V$-modules given in Huang's paper. Then we define right $V$-modules and $V$-bimodules that reflect the noncommutative nature of $V$. When $V$ satisfies a condition on the order of poles of the correlation function (which we call pole-order condition), we prove that the rationality of products of two vertex operators implies the rationality of products of any numbers of vertex operators. Also, the rationality of iterates of any numbers of vertex operators is established, and is used to construct the opposite MOSVA $V^{op}$ of $V$. It is proved here that right (resp. left) $V$-modules are equivalent to left (resp. right) $V^{op}$-modules. Using this equivalence, we prove that if $V$ and a grading-restricted left $V$-module $W$ is endowed with a M\"obius structure, then the graded dual $W'$ of $W$ is a right $V$-module. This proof is the only place in this paper that needs the grading-restriction condition. Also, this result is generalized to not-grading-restricted modules under a strong pole-order condition that is satisfied by all existing examples of MOSVAs and modules. 
\end{abstract}

\maketitle

\section{Introduction}

Vertex (operator) algebras are algebraic structures formed by (meromorphic) vertex operators. In mathematics, they arose naturally in the study of representations of infinite-dimensional Lie algebras and the Monster group (see \cite{FLM} and \cite{B}). In physics, they arose in the study of two-dimensional conformal field theory (2d CFT hereafter, see \cite{BPZ} and \cite{MS}). One of the most important properties of the vertex operators for a vertex (operator) algebra is the commutativity, which plays an important role in the study of these algebras and their representation theory. Mathematically, the commutativity, especially the equivalent commutator formula, makes it possible to use the Lie-theoretic methods to study vertex (operator) algebras and modules. Many results are proved based on the commutativity. In physical terms, vertex operators for a vertex (operator) algebra or a module correspond to fields of a special kind: meromorphic fields. The commutativity of vertex operators is closely related to the locality of meromorphic fields in two-dimensional conformal field theory. This commutativity is one of the most important reasons for the success of the mathematical construction of 2d CFT using the vertex (operator) algebras, its modules and the intertwining operators among the modules. 

However, if we want to use vertex-algebraic methods to study quantum field theories in general, the commutativity might not hold even for meromorphic fields. One important class of quantum field theories is the nonlinear $\sigma$-model with the target manifold being a Riemannian manifold. If we want to realize certain differential operators on the manifold as components of some vertex operators, then these vertex operators cannot be commutative. 

On the other hand, vertex (operator) algebras also have associativity, which is even more fundamental. In physical terms, associativity of vertex operators can be viewed as a strong form of the operator product expansion (OPE hereafter) of meromorphic fields. And the OPE of fields is expected to hold for all quantum field theories. This is one of the motivations for studying algebraic structures of suitable vertex operators that have associativity but not necessarily commutativity. In 2003, Huang and Kong introduced and constructed open-string vertex algebras in \cite{HK-OSVA}. In 2012, Huang introduced the notion of meromorphic open-string vertex algebras in \cite{H-MOSVA}, a special case of open-string vertex algebra for which the correlation functions are rational functions. 

Our motivation of studying meromorphic open-string vertex algebras (MOSVAs hereafter) are the following: first, just as vertex (operators) algebras can be viewed as analogues of commutative associative algebras, MOSVAs can be viewed as analogues of associative algebras that are not necessarily commutative. In particular, all vertex (operator) algebras are MOSVAs. So all the results for MOSVA also hold for vertex (operator) algebras. Since all correlation functions are rational functions, it is easier to deal with issues related to convergence and analytic extensions for MOSVAs than general open-string vertex algebras. 

In 2012, Huang also constructed an example of MOSVA using parallel sections of tensor products of tangent bundles on any fixed Riemannian manifold (see \cite{H-MOSVA-Riemann}). The explicit computation in the case of the spheres will be given in \cite{Q-Ex}. More importantly, Huang constructed modules generated by eigenfunctions of Laplacian operator. In physics, the eigenfunctions correspond to quantum states of a particle, which can be viewed as a degenerated form of a string. Elements of the MOSVA modules generated by eigenfunctions can be viewed as suitable string-theoretic excitations of the particle states. It is Huang's idea that the MOSVAs constructed from Riemannian manifolds, together with modules generated by Laplacian eigenfunctions and the still-to-be-defined intertwining operators among these modules may lead to a mathematical construction of the quantum two-dimensional nonlinear $\sigma$-model. Huang also hopes that this will shed lights on the four-dimensional Yang-Mills theory, which, though much more difficult, is indeed analogous to the two-dimensional nonlinear $\sigma$-model whose target manifold is a Lie group. 

Another motivation for studying MOSVA is the cohomological criterion for the reductivity
Huang and the author proved in 2015. In 2010, Huang developed the cohomology theory of grading-restricted vertex algebras in \cite{Hcoh} and \cite{H1st-sec-coh}, which can be viewed as an analogue to the Harrison cohomology for commutative associative algebras. In fact, Huang already introduced in \cite{Hcoh} a cohomology analogous to the Hochschild cohomology for associative algebras and the criterion we found is expressed in terms of this cohomology, not the one analogous to the Harrison cohomology. We proved that, if all first cohomologies of a grading-restricted vertex algebra in this cohomology theory vanish, then every module of finite length satisfying a certain composable condition is completely reducible. Just as Hochschild cohomology is for all associative algebras that are not necessarily commutative, a cohomology theory analogous to Hochschild cohomology for MOSVA should be established. The cohomological criterion we proved for vertex algebras also generalizes to MOSVAs. The details of the cohomology theory will be given in \cite{Q-Coh}. And the proof of the cohomological criterion of reductivity will be given in \cite{HQ-Red}. 

This paper further develops the theory of MOSVA to provide a foundation for \cite{Q-Ex}, \cite{Q-Coh}, \cite{HQ-Red} and future research. The sections are organized as follows:

In Section 2 we recall the definitions of MOSVA and left module defined in \cite{H-MOSVA} and give the definitions of right module and bimodule for MOSVAs. Roughly speaking, the right module is defined in such a way that all the properties for the intertwining operator of type $\binom{W}{W V}$ given in [FHL] that still make sense hold. The bimodules are defined as a vector space with both a left module structure and a right module structure satisfying certain compatibility conditions. 

In Section 3, we discuss a pole-order condition that simplifies the verifications of the rationality axiom. With this pole-order condition, it suffices to verify the rationality axiom for the products of two vertex operators. For vertex (operator) algebras, such results were proved using commutativity (See Sec. 2.6 of \cite{FHL}). It is somehow surprising to have an argument that only involves associativity. Also, we express the pole-order condition using formal series and weak associativity, thus obtain a sufficient condition that can be used to verify the axioms for MOSVAs and modules in terms of formal variables. 

In Section 4, we discuss the rationality of iterates of any number of vertex operators and specify the region of convergence. In the discussion we do not assume the module $W = \coprod\limits_{m\in \C} W_{[m]}$ to be grading-restricted, i.e., $\dim W_{[m]}$ can be infinite for some $m$. So instead of using the algebraic completion $\overline{W} = \prod\limits_{m\in \C} W_{[m]}$, we need to use the much larger space $\arc{W} = \prod\limits_{m\in\C} W_{[m]}^{**}$. Complex analysis plays an important role in dealing with the convergence issues on this space. Some exposition is provided for the convenience of the reader. The discussion is for left modules only, but the argument and conclusion easily generalize to right modules and bimodules. 

In Section 5, for a fixed MOSVA, we define the skew-symmetry opposite vertex operators and opposite MOSVAs, which are analogous to opposite multiplications and opposite rings in classical ring theory. We prove that left modules for the opposite MOSVA is equivalent to right modules for the MOSVA, and right modules for the opposite MOSVA is equivalent to left modules for the MOSVA. The main technical part is to verify the rationality axiom. The results in Section 4 are needed in the proof.  

In Section 6, we define the notion of a M\"obius MOSVA, a noncommutative generalization of the M\"obius vertex algebras (see \cite{FHL}). We prove that for a M\"obius MOSVA $V$ and a grading-restricted M\"obius left $V$-module $W=\coprod_{m\in \C} W_{[m]}$, the graded dual $W' = \coprod_{m\in \C} W_{[m]}^*$ of $W$, with appropriately defined vertex operator, is a M\"obius right $V$-module. The result generalizes to non-grading-restricted modules if the vertex operators of the module $W$ satisfies a stronger pole-order condition. The proof uses the theory of opposite MOSVAs developed in Section 5. 

\noindent \textbf{Acknowledgement.} This paper is written under the guidance of Yi-Zhi Huang. The author is very grateful to his patience in the discussions, his enthusiasm in this topic, and his tolerance to my occasional silly questions. The author also thanks James Lepowsky for his helpful comments. 

\section{Definitions and Immediate Consequences}

\subsection{Definition of the MOSVA}

We first recall the notion of meromorphic open-string vertex algebra given in \cite{H-MOSVA}. By convention, $x$ is understood as a formal variable while $z$ is understood as a complex variable. 

\begin{defn}\label{DefMOSVA}
{\rm A {\it meromorphic open-string vertex algebra} (hereafter MOSVA) is a $\Z$-graded vector space 
$V=\coprod_{n\in\Z} V_{(n)}$ (graded by {\it weights}) equipped with a {\it vertex operator map}
\begin{eqnarray*}
   Y_V:  V\otimes V &\to & V[[x,x^{-1}]]\\
	u\otimes v &\mapsto& Y_V(u,x)v
  \end{eqnarray*}
and a {\it vacuum} $\one\in V$, satisfying the following axioms:
\begin{enumerate}
\item Axioms for the grading:
\begin{enumerate}
\item {\it Lower bound condition}: When $n$ is sufficiently negative,
$V_{(n)}=0$.
\item {\it $\d$-commutator formula}: Let $\d_{V}: V\to V$
be defined by $\d_{V}v=nv$ for $v\in V_{(n)}$. Then for every $v\in V$
$$[\d_{V}, Y_{V}(v, x)]=x\frac{d}{dx}Y_{V}(v, x)+Y_{V}(\d_{V}v, x).$$
\end{enumerate}

\item Axioms for the vacuum: 
\begin{enumerate}
\item {\it Identity property}: Let $1_{V}$ be the identity operator on $V$. Then
$Y_{V}(\mathbf{1}, x)=1_{V}$. 
\item {\it Creation property}: For $u\in V$, $Y_{V}(u, x)\mathbf{1}\in V[[x]]$ and 
$\lim_{x\to 0}Y_{V}(u, x)\mathbf{1}=u$.
\end{enumerate}

\item {\it $D$-derivative property and $D$-commutator formula}:
Let $D_V: V\to V$ be the operator
given by
$$D_{V}v=\lim_{x\to 0}\frac{d}{dx}Y_{V}(v, x)\one$$
for $v\in V$. Then for $v\in V$,
$$\frac{d}{dx}Y_{V}(v, x)=Y_{V}(D_{V}v, x)=[D_{V}, Y_{V}(v, x)].$$

\item {\it Rationality}: Let $V'=\coprod_{n\in \Z}V_{(n)}^*$ be the graded dual of $V$. 
For $u_1, \cdots, u_n, v\in V, v'\in V'$, the series
$$\langle v', Y_V(u_1, z_1)\cdots Y_V(u_n, z_n)v\rangle$$
converges absolutely when $|z_1|>\cdots>|z_n|>0$ to a rational function in $z_1,\cdots, z_n$, 
with the only possible poles at $z_i=0, i=1, ... , n$ and $z_i=z_j, 1\leq i\neq j\leq n$. 
For $u_1, u_2, v\in V$ and $v' \in V'$, the series
$$\langle v', Y_V(Y_V(u_1,z_1-z_2)u_2,z_2)v\rangle$$
converges absolutely when $|z_2|>|z_1-z_2|>0$ to a rational function with the only possible poles at 
$z_1=0, z_2=0$ and $z_1=z_2$.

\item {\it Associativity}: For $u_{1}, u_{2}, v\in V$ and
$v'\in V'$, we have 
$$\langle v', Y_{V}(u_{1}, z_{1})Y_{V}(u_{2}, z_{2})v\rangle=\langle v', Y_{V}(Y_{V}(u_{1}, z_{1}-z_{2})u_{2}, z_{2})v\rangle$$
when $|z_{1}|>|z_{2}|>|z_{1}-z_{2}|>0$.
\end{enumerate}  }
We denote the MOSVA just defined by $(V, Y_V, \one)$ or simply by $V$ when there is no confusion.
\end{defn}

\begin{defn}
A meomorphic open-string vertex algebra $V$ is said to be \textit{grading-restricted} if 
$\dim V_{(n)}<\infty$ for $n\in \Z$.
\end{defn}

\begin{rema}
Most of the current-existing examples of MOSVA are grading-restricted. But in this paper, all the conclusion hold without this assumption. \end{rema}


\begin{rema}
If in addition, $V$ satisfies commutativity, namely, for every $u_1, u_2, v\in V, v' \in V'$
$$\langle v', Y_V(u_1, z_1)Y_V(u_2, z_2)v\rangle$$
converges absolutely when $|z_1|>|z_2|> 0$ to the same rational function that 
$$\langle v', Y_V(Y_V(u_1, z_1-z_2)u_2, z_2)v\rangle$$
converges to when $|z_2| > |z_1-z_2|> 0$, it was proved in \cite{FHL} that in this case the Jacobi identity for vertex algebras holds and $V$ is a vertex algebra with lower bounded $\Z$-grading. So a MOSVA can be treated as a noncommutative generalization to a vertex algebra. All vertex operator algebras are MOSVAs.
\end{rema}

Axioms 1, 2 and 3 make it possible to carry over some facts of vertex algebras to MOSVA: 

\begin{prop}\label{ImmediateFacts}
Let $V$ be a MOSVA. Then
\begin{enumerate}
\item For $u\in V$, $Y_V(u,x)$ can be regarded as a formal series in $\text{End}(V)[[x,x^{-1}]]$
$$Y_V(u, x)= \sum_{n\in\Z} (Y_V)_n(u) x^{-n-1}$$
where $(Y_V)_n(u): V \to V$ is a linear map for every $n\in \Z$. If $u$ is homogeneous, then $(Y_V)_n(u)$ is a map of weight $\wt u - n- 1$. 
\item For fixed $u, v\in V$, $Y_V(u, x)v$ is lower truncated, i.e, the series has at most finitely many negative powers of $x$. 
\item For $u\in V$, 
$$Y_V(u, x)\one = e^{xD_V} u$$ 
\item $D$-conjugation property: for $u\in V$, 
$$Y_V(u, x+y) = Y_V(e^{yD_V}u, x) = e^{yD_V} Y_V(u, x) e^{-yD_V},$$
in $\text{End}(V)[[x,x^{-1}, y]]$. 
\item $\d$-conjugation property: for $u\in V$, 
$$  e^{y\d_V} Y_V(u, x) e^{-y \d_V} = Y_V(e^{y\d_V} u, xy)$$
in $\text{End}(V)[[x,x^{-1}, y, y^{-1}]]$. 
\end{enumerate}
\end{prop}

\begin{rema}\label{Taylor}
In the statement of the $D$-conjugation property, as in \cite{FHL} and \cite{LL}, the series $Y(u, x+y)$ is a series with two variables $x$ and $y$, obtained by expanding all the powers of $x+y$ as power series in $y$. 
\end{rema}

\begin{proof}
\begin{enumerate}
\item Follows from the linearity of $Y_V(u,x)v$ in both $u$ and $v$, and the $\d$-commutator formula. 
\item When $u, v$ are homogeneous, the coefficient $(Y_V)_n(u)v$ of $x^{-n-1}$ in $Y_V(u,x)v$ is also homogeneous of weight $m = \wt u + \wt v - n - 1$. As $n$ gets sufficient large, $m$ becomes sufficiently negative and by the lower bound condition, $V_{(m)}=0$. So $(Y_V)_n(u)v = 0$ when $n$ gets sufficiently large, hence the series $\sum\limits_{n\in \Z} (Y_V)_n(u)vx^{-n-1}$ is lower truncated. 
\item Use the $D$-derivative property and induction, it is easy to show for $n=0, 1, ...$, 
$$D_V^n v = \lim_{x\to 0} \frac{d^n}{dx^n} Y(v,x)\one$$
So $Y_V(v,x)\one$, as a power series, has $D_V^n v/ n!$ as the coefficient of $x^n$. Hence 
$$Y_V(v,x)\one = \sum_{n=0}^\infty \left(\frac 1 {n!} D_V^n v\right)x^n = e^{xD_V} v$$
\item The first equality follows from the $D_V$-derivative property. The second equality follows from the exponentiation of the $D_V$-commutator formula. 
\item This is a consequence of the $\d_V$-commutator formula. 
\end{enumerate}
\end{proof}

\begin{rema}
Recall that a vertex algebra with lower bounded $\Z$-grading can be defined using the following duality axioms: for every $v'\in V, u_1, u_2, v\in V$, the following series
$$\langle v', Y_V(u_1, z_1)Y_V(u_2,z_2)v\rangle, |z_1|>|z_2|>0$$
$$\langle v', Y_V(u_2, z_2)Y_V(u_1,z_1)v\rangle, |z_2|>|z_1|>0$$
$$\langle v', Y_V(Y_V(u_1, z_1-z_2)u_2,z_2)v\rangle, |z_2|>|z_1-z_2|>0$$
converge absolutely to a common rational function with the only possible poles at $z_1=0,z_2=0,z_1=z_2$. We don't have to assume the rationality of products of more than 2 vertex operators because this can be obtained with the commutativity (See \cite{FHL}, Section 2.6 for details). With the absence of commutativity, such rationality won't hold unless one imposes extra conditions. We will investigate these conditions in Section 3. 
\end{rema}

\begin{rema}\label{FormFact}
For the MOSVA defined in Definition \ref{DefMOSVA}, some familiar facts in terms of formal series in general do \textit{not} hold. For example, there might not exist $p_1\in \N$, such that for some $p_2, p_{12}\in \N$,
$$x_1^{p_1}x_2^{p_2}(x_1-x_2)^{p_{12}}Y_V(u_1, x_1)Y_V(u_2, x_2)v\in V[[x_1, x_2]]$$
So in general, it is \textit{not} true that $Y_V(u_1, x_1)Y_V(u_2,x_2)v$ is the expansion of a power series localized at the multiplicative set generated by $x_1, x_2$ and $x_1-x_2$ 
with all the negative powers of $(x_1-x_2)$ expanded as power series in $x_2$. 
In fact, if we want $Y_V(u_1, x_1)Y_V(u_2, x_2)v$ to expanded from such a localized power series, we need to make sure that there exists integers $p_1, p_2$ and $p_{12}$ such that 
$z_1^{p_1}z_2^{p_2}(z_1-z_2)^{p_{12}}\langle v', Y_V(u_1, z_1)Y_V(u_2,z_2)v\rangle$
converges absolutely to polynomial function in $z_1, z_2$ for every $v'\in V'$. While one can find $p_2$ and $p_{12}$ that does not depend on $v'$, Definition \ref{DefMOSVA} does not tell us if one can find such $p_1$ .  

We will come back to this issue in Section 3 and systematically discuss conditions that should be assumed in addition to those in Definition \ref{DefMOSVA} for the formal variable approach to work. 

\end{rema}

\begin{rema}
Note also that we only assume the rationality of iterates of two vertex operators in the definition. The rationality of iterates of any numbers of vertex operators can be proved based on these assumptions. The details will be given in Section 4. 
\end{rema}


\subsection{Left modules for MOSVAs}

The notion of left modules for a meromorphic open-string vertex algebra was introduced in \cite{H-MOSVA}. The philosophy is similar to the modules for vertex algebras: all the defining properties of a MOSVA that make sense hold. 

\begin{defn}\label{DefMOSVA-L}
Let $(V, Y_{V}, \one)$ be a meromorphic open-string vertex algebra.
A \textit{left $V$-module} is a $\C$-graded vector space 
$W=\coprod_{m\in \C}W_{[m]}$ (graded by \textit{weights}), equipped with 
a \textit{vertex operator map}
\begin{eqnarray*}
Y_W^L: V\otimes W & \to & W[[x, x^{-1}]]\\
u\otimes w & \mapsto & Y_W^L(u, x)w,
\end{eqnarray*}
an operator $\d_{W}$ of weight $0$ and 
an operator $D_{W}$ of weight $1$, satisfying the 
following axioms:
\begin{enumerate}

\item Axioms for the grading: 
\begin{enumerate}
\item \textit{Lower bound condition}:  When $\text{Re}{(m)}$ is sufficiently negative,
$W_{[m]}=0$. 
\item  \textit{$\mathbf{d}$-grading condition}: for every $w\in W_{[m]}$, $\d_W w = m w$.
\item  \textit{$\mathbf{d}$-commutator formula}: For $u\in V$, 
$$[\mathbf{d}_{W}, Y_W^L(u,x)]= Y_W^L(\mathbf{d}_{V}u,x)+x\frac{d}{dx}Y_W^L(u,x).$$
\end{enumerate}

\item The \textit{identity property}:
$Y_W^L(\one,x)=1_{W}$.

\item The \textit{$D$-derivative property} and the  \textit{$D$-commutator formula}: 
For $u\in V$,
\begin{eqnarray*}
\frac{d}{dx}Y_W^L(u, x)
&=&Y_W^L(D_{V}u, x) \\
&=&[D_{W}, Y_W^L(u, x)].
\end{eqnarray*}

\item \textit{Rationality}: For $u_{1}, \dots, u_{n}\in V, w\in W$
and $w'\in W'$, the series 
$$
\langle w', Y_W^L(u_{1}, z_1)\cdots Y_W^L(u_{n}, z_n)w\rangle
$$
converges absolutely 
when $|z_1|>\cdots >|z_n|>0$ to a rational function in $z_{1}, \dots, z_{n}$
with the only possible poles at $z_{i}=0$ for $i=1, \dots, n$ and $z_{i}=z_{j}$ 
for $i\ne j$. For $u_{1}, u_{2}\in V, w\in W$
and $w'\in W'$, the series 
$$
\langle w', Y_W^L(Y_{V}(u_{1}, z_1-z_{2})u_{2}, z_2)w\rangle
$$
converges absolutely when $|z_{2}|>|z_{1}-z_{2}|>0$ to a rational function
with the only possible poles at $z_{1}=0$, $z_{2}=0$ and $z_{1}=z_{2}$. 

\item \textit{Associativity}: For $u_{1}, u_{2}\in V, w\in W$, 
$w'\in W'$, 
$$
\langle w', 
Y_W^L(u_{1},z_1)Y_W^L(u_{2},z_2)w\rangle
=
\langle w', 
Y_W^L(Y_{V}(u_{1},z_{1}-z_{2})u_{2},z_2)w\rangle
$$
when  $|z_{1}|>|z_{2}|>|z_{1}-z_{2}|>0$. 
\end{enumerate} 
We denote the left $V$-module just defined by $(W, Y_W^L, \d_{W}, D_{W})$ or simply $W$ when there is no confusion. 
\end{defn}

\begin{defn}
A left $V$-module is said to be \textit{grading-restricted} if 
$\dim W_{[m]}<\infty$ for every $m\in \C$. 
\end{defn}

\begin{rema}
Throughout the discussion in this paper, we will \textit{not} assume the $V$-modules to be grading-restricted, except for Theorem \ref{W'Module} in Section 6. This theorem can still be generalized to non-grading-restricted modules under certain additional conditions. See Theorem \ref{W'Module-1} for details. 
\end{rema}

\begin{rema}\label{numdegree}
If we let 
$$\frac{f(z_1, ..., z_n)}{\prod\limits_{i=1}^n z_i^{p_i}\prod\limits_{1\leq i < j \leq n} (z_i-z_j)^{p_{ij}}}$$
be the rational function determined by the series 
$$\langle w', Y_W^L(u_1, z_1)\cdots Y_W^L(u_n, z_n)w\rangle,$$
then for homogeneous $u_1, ..., u_n\in V, w\in W, w'\in W'$, we can explicitly compute the total degree of the homogeneous polynomial $f(z_1, ..., z_n)$ in terms of the weights and $p_i, p_{ij}$'s. We start by expanding the series as
$$\sum_{k_1, ..., k_n} \langle w', (Y_W^L)_{k_1}(u_1) \cdots (Y_W^L)_{k_n} (u_n) w\rangle z_1^{-k_1-1}\cdots z_n^{-k_n-1} $$
then the coefficients are nonzero only when
$$\wt w' = \wt u_1 - k_1 - 1 + \cdots + \wt u_n - k_n - 1 + \wt w$$
In particular, 
$$\text{Re}(\wt w') = \wt u_1 - k_1 - 1 + \cdots + \wt u_n - k_n - 1 + \text{Re}(\wt w)$$
Thus 
\begin{align*}
\deg f & = \sum_{i=1}^n p_i + \sum_{1\leq i < j \leq n} p_{ij} + (-k_1-1-k_2-1-\cdots - k_n -1 )\\
& = \sum_{i=1}^n p_i + \sum_{1\leq i < j \leq n} p_{ij} + \text{Re}(\wt w') - \sum_{i=1}^n \wt u_i - \text{Re}(\wt w)
\end{align*}
In particular, when there are only two vertex operators, the total degree of the homogeneous polynomial in the numerator is just
$$p_1 + p_2 + p_{12} + \text{Re}(\wt w') - \wt u_1 - \wt u_2 - \text{Re}(\wt w)$$
\end{rema}

\begin{prop}\label{ImmediateFacts-L}
Let $V$ be a MOSVA and $W$ be a left $V$-module. Then we have:
\begin{enumerate}
\item For $u\in V$, $Y_W^L(u,x)$ can be regarded as a formal series in $\text{End}(W)[[x,x^{-1}]]$
$$Y_W^L(u, x)= \sum_{n\in\Z} (Y_W^L)_n(u) x^{-n-1}$$
where $(Y_W^L)_n(u):W \to W$ is a linear map for every $n\in \Z$. If $u$ is homogeneous, then $(Y_W^L)_n(u)$ is a map of weight $\wt u - n- 1$. 
\item For fixed $u\in V, w\in W$, $Y_W^L(u, x)w$ is lower truncated, i.e, there are at most finitely many negative powers of $x$. 
\item $D$-conjugation property: for $u\in V$, 
$$Y_W^L(u, x+y) = Y_V(e^{yD_V}u, x) = e^{yD_W} Y_W^L(u, x) e^{-yD_W},$$
in $\text{End}(W)[[x,y, x^{-1}]]$. 
\item $\d$-conjugation property: for $u\in V$, 
$$  e^{y\d_W} Y_W^L(u, x) e^{-y \d_W} = Y_W^L(e^{y\d_W} u, xy)$$
in $\text{End}(W)[[x,x^{-1}, y, y^{-1}]]$. 
\end{enumerate}
\end{prop}
\begin{proof}
Similar to the argument of Proposition \ref{ImmediateFacts}. For the second statement, use the fact that that $W_{[m]}= 0$ when Re $m << 0$. 
\end{proof}


\subsection{Right modules for MOSVAs}

\begin{defn}
Let $(V, Y_{V}, \one)$ be a meromorphic open-string vertex algebra.
A \textit{right module for $V$} 
is a $\C$-graded vector space 
$W=\coprod_{m\in \C}W_{[m]}$ (graded by \textit{weights}), equipped with 
a \textit{vertex operator map}
\begin{eqnarray*}
Y_W^R: W\otimes V&\to& W[[x, x^{-1}]]\\
w\otimes u&\mapsto& Y_W^R(w, x)u,
\end{eqnarray*}
an operator $\d_{W}$ and 
an operator $D_{W}$ of weight $1$, satisfying the 
following axioms:
\begin{enumerate}

\item Axioms for the grading
\begin{enumerate}
\item \textit{Lower bound condition}:  When $\text{Re }{m}$ is sufficiently negative, $W_{[m]}=0$. 
\item \textit{$\mathbf{d}$-grading condition}: for every $w\in W_{[m]}, \d_W w = m w$.
\item \textit{$\d$-commutator formula}: For $w\in W$, 
$$\d_{W}Y_W^R(w,x)-Y_W^R(w,x)\d_{V}= Y_W^R(\d_{W}w,x)+x\frac{d}{dx}Y_W^R(w,x).$$
\end{enumerate}

\item The \textit{Creation property}: For $w\in W$, $Y_W^R(w,x)\one\in W[[x]]$ and 
$\lim\limits_{x\to 0}Y_W^R(w,x)\one=w$.

\item The \textit{$D$-derivative property} and the  \textit{$D$-commutator formula}: 
For $u\in V$,
\begin{eqnarray*}
\frac{d}{dx}Y_W^R(w, x)
&=&Y_W^R(D_{W}w, x) \\
&=&D_{W}Y_W^R(w, x)-Y_W^R(w, x)D_{V}.
\end{eqnarray*}

\item \textit{Rationality}: For $u_{1}, \dots, u_{n}\in V, w\in W$
and $w'\in W'$, the series 
$$
\langle w', Y_W^R(w, z_1)Y_{V}(u_{1}, z_2)\cdots Y_{V}(u_{n-1}, z_n)u_{n}\rangle
$$
converges absolutely 
when $|z_1|>\cdots >|z_n|>0$ to a rational function in $z_{1}, \dots, z_{n}$
with the only possible poles at $z_{i}=0$ for $i=1, \dots, n$ and $z_{i}=z_{j}$ 
for $i\ne j$. For $u_{1}, u_{2}\in V, w\in W$
and $w'\in W'$, the series 
$$
\langle w', Y_W^R(Y_W^R(w, z_1-z_{2})u_{1}, z_2)u_{2}\rangle
$$
converges absolutely when $|z_{2}|>|z_{1}-z_{2}|>0$ to a rational function
with the only possible poles at $z_{1}=0$, $z_{2}=0$ and $z_{1}=z_{2}$. 

\item \textit{Associativity}: For $u_{1}, u_{2}\in V, w\in W$, 
$w'\in W'$, 
$$
\langle w', 
Y_W^R(w,z_1)Y_{V}(u_{1},z_2)u_{2}\rangle
=
\langle w', 
Y_W^R(Y_W^R(w,z_{1}-z_{2})u_{1},z_2)u_{2}\rangle
$$
when  $|z_{1}|>|z_{2}|>|z_{1}-z_{2}|>0$. 
\end{enumerate} 

A right $V$-module is said to be \textit{grading-restricted} if 
$\dim W_{[n]}<\infty$ for $n\in \C$. 

When there is no confusion, we also denote the right $V$-module just defined by $(W, Y_W^R, \d_{W}, D_{W})$ or simply $W$. 
\end{defn}

\begin{rema}
The right module is defined with the following philosophy: all the properties of intertwining operators of type $\binom{W}{WV}$ that make sense hold. With such a formulation, all the issues on convergence can be analyzed similarly as the vertex operators in a left $V$-module. 
\end{rema}

\begin{prop}\label{ImmediateFacts-R}
Let $V$ be a MOSVA and $W$ a right $V$-module. Then we have: 
\begin{enumerate}
\item For $u\in V$, $Y_W^R(\cdot,x)u$ can be regarded as a formal series in $\text{End}(W)[[x,x^{-1}]]$
$$Y_W^R(\cdot, x)u= \sum_{n\in\Z} (Y_W^R)_n(\cdot)u x^{-n-1}$$
where $(Y_W^R)_n(\cdot)u: W \to W$ is a linear map for every $n\in \Z$. If $u$ is homogeneous, then $(Y_W^R)_n(\cdot)u$ is a map of weight $\wt u - n- 1$. 
\item For fixed $u\in V$ and $w\in W$, $Y_W^R(w, x)u$ is lower truncated, i.e, there are at most finitely many negative powers of $x$. 
\item For $w\in W$, 
$$Y_W^R(w, x)\one = e^{xD_W} w$$ 
\item $D$-conjugation property: for $w\in V$, 
$$Y_W^R(w, x+y) = Y_W^R(e^{yD_W}w, x) = e^{yD_W} Y_W^R(w, x) e^{-yD_V},$$
in $\text{End}(W)[[x,y, x^{-1}]]$. 
\item $\d$-conjugation property: for $u\in V$, 
$$  e^{y\d_W} Y_W^R(w, x) e^{-y \d_V} = Y_W^R(e^{y\d_W} w, xy)$$
in $\text{End}(W)[[x,x^{-1}, y, y^{-1}]]$. 
\end{enumerate}
\end{prop}
\begin{proof}
The arguments for (1), (2), (4) and (5) are similar to those for Proposition \ref{ImmediateFacts}. To see (3), one first note that 
$$D_W w = \lim_{x\to 0} Y_W^R(D_Ww, x)\one = \lim_{x\to 0} \frac d{dx} Y_W^R(w, x)$$
(the first equality follows from the creation property, the second from $D$-derivative property), then apply the arguments in Proposition \ref{ImmediateFacts}. 
\end{proof}


\subsection{Bimodules for MOSVAs}

\begin{defn}
Let $(V, Y_V, \one)$ be a meromorphic open-string vertex algebra. Roughly speaking, a \textit{$V$-bimodule} is a vector space equipped with a left $V$-module structure and a right $V$-module structure such that these two strutcure are compatible. 
A \textit{$V$-bimodule} is a $\C$-graded vector space 
$$W=\coprod_{n\in \C}W_{[n]}$$
equipped with a \textit{left vertex operator map}
\begin{eqnarray*}
Y_{W}^{L}: V\otimes W&\to& W[[x, x^{-1}]]\\
u\otimes w&\mapsto& Y_{W}^{L}(u, x)v,
\end{eqnarray*}
a \textit{right vertex operator map}
\begin{eqnarray*}
Y_{W}^{R}: W\otimes V&\to& W[[x, x^{-1}]]\\
w\otimes u&\mapsto& Y_{W}^{R}(w, x)u,
\end{eqnarray*}
and linear operators $d_W, D_W$ on $W$ satisfying the following conditions.
\begin{enumerate}

 \item $(W, Y_W^L, \d_W, D_W)$ is a left $V$-module.

 \item $(W, Y_W^R, \d_W, D_W)$ is a right $V$-module.

\item \textit{Compatibility}:
\begin{enumerate}
\item 
\textit{Rationality of left and right vertex operator maps}: For $u_1, ..., u_n, u_{n+1}, ..., u_{n+m}\in V$, $w\in W$, the series
$$\langle w', Y_W^L(u_1, z_1) \cdots Y_W^L(u_n, z_n) Y_W^R(w, z_{n+1}) Y_V(u_{n+1}, z_{n+2})\cdots Y_V(u_{n+m-1}, z_{n+m})u_{n+m}\rangle$$
converges absolutely when $|z_1|>|z_2|>\cdots >|z_n| > |z_{n+1}| > \cdots > |z_{n+m}|> 0$ to a rational function in $z_1, ..., z_n, z_{n+1}, ..., z_{n+m}$. 
\item \textit{Associativity for left and right vertex operator maps}: For $u, v\in V$, $w\in W$ and $w'\in W'$, the series
$$\langle w', Y_W^L(u, z_1)Y_W^R(w,z_2)v\rangle$$
$$\langle w', Y_W^R(Y_W^L(u,z_1-z_2)w, z_2)v\rangle$$
converges absolutely when $|z_1|>|z_2|>0$ and $|z_2|>|z_1-z_2|>0$,  respectively, to a common rational function in 
$z_1$ and $z_2$ with the only possible poles at $z_1, z_2=0$ and $z_1 = z_2$.
\end{enumerate}
\end{enumerate}
\end{defn}

The $V$-bimodule just defined is denoted 
by  $(W, Y_W^L, Y_W^R, \d_W, D_W)$ or simply by $W$ when there is no confusion

\begin{rema}
It is possible to generalize the definition to allow the left and right module structure on $W$ to have different $\d$ and $D$ operators. We decide not to discuss such a generalization in this paper as we have no motivations for now. 
\end{rema}

\begin{rema}
Note that the compatibility conditions implies the rationality of products of any numbers of vertex operators in the left and right module structure. 
\end{rema}


\section{Formal Variable Approach and the Pole-Order Condition}

In this section we introduce a pole-order condition that is satisfied for all the current-existing examples of MOSVAs and modules. With this pole-order condition, we prove that the rationality of products of two vertex operators implies the rationality of products of any numbers of vertex operators. As an application, we introduce the weak associativity relation with pole-order condition, which is a sufficient condition formulated in terms of formal series for the rationality and associativity axioms 

\subsection{Formal variable approach and weak associativity} \label{FormFact-1}

Let $V$ be a MOSVA. We have seen in Remark \ref{FormFact}  that if we want 
$$x_1^{p_1}x_2^{p_2}(x_1-x_2)^{p_{12}}Y_V(u_1, x_1)Y_V(u_2, x_2)v\in V[[x_1, x_2]]$$
for some integers $p_1, p_2, p_{12}$, it is necessary that for every $v'\in V'$, the order of the pole $z_1=0$ is bounded above by a number that is independent of $v'$. Now we give a brief argument to show that this is sufficient. 

As a formal series, both $Y_V(u_2, x_2)v$ and $Y_V(u_1, x_0)u_2$ are lower truncated. So there exists an integer $p_2$ that depends only on $u_2$ and $v$, and an integer $p_{12}$ that depends only on $u_1$ and $u_2$, such that $$x_2^{p_2}Y_V(u_2, x_2)v \in V[[x_2]], x_0^{p_{12}}Y_v(u_1, x_0)u_2 \in V[[x_0]]$$ 
Thus the formal series
$$ x_2^{p_2}\langle v', Y_V(u_1, x_1)Y_V(u_2, x_2)v\rangle$$
has no negative powers of $x_2$, and the formal series
$$ x_0^{p_{12}} \langle v', Y_V(Y_V(u_1, x_0)u_2, x_2)v\rangle$$
has no negative powers of $x_0$. Then from associativity, we know that the series
$$ z_2^{p_2}(z_1-z_2)^{p_{12}}\langle v', Y_V(u_1, z_1)Y_V(u_2, z_2)v\rangle $$
converges absolutely to a rational function that has the only possible poles at $z_1=0$. 
Now we use the condition on the pole $z_1=0$. Let $p_1$ be an upper bound of the order that does not depend on $v'$, we see that
$$ z_1^{p_1}z_2^{p_2}(z_1-z_2)^{p_{12}}\langle v', Y_V(u_1, z_1)Y_V(u_2, z_2)v\rangle $$
gives a polynomial function on $\C$. Thus as formal series, 
$$x_1^{p_1}x_2^{p_2}(x_1-x_2)^{p_{12}}\langle v', Y_V(u_1, x_1)Y_V(u_2, x_2)v\rangle \in \C[x_1, x_2]$$
for every $v'\in V$. Therefore, as a formal series in $V[[x_1, x_1^{-1}, x_2, x_2^{-1}]]$, 
$$x_1^{p_1}x_2^{p_2}(x_1-x_2)^{p_{12}}Y_V(u_1, x_1)Y_V(u_2, x_2)v$$
has no negative powers of $x_1, x_2$. 

Also observe in this case that for every $u_1, u_2, v\in V$, for the $p_1$ chosen as above, we see from associativity that 
\begin{equation}\label{WeakAssoc}
(x_0+x_2)^{p_1}Y_V(u_1, x_0+x_2)Y_V(u_2, x_2)v = (x_0+x_2)^{p_1}Y_V(Y_V(u_1, x_0)u_2, x_2)v 
\end{equation}
as formal series in $V[[x_0, x_0^{-1}, x_2, x_2^{-1}]]$. We call Equation (\ref{WeakAssoc}) the \textit{weak associativity}.

In particular, this shows the following: if we use localized formal power series with coefficients in $W$ to formulate the rationality axiom (requiring that after multiplied by certain powers of $x_i, x_j, (x_i-x_j), 1\leq i < j \leq n$, the formal series $Y_V(u_1, x_1)\cdots Y_V(u_n, x_n)v$ has no negative powers) and associativity axiom (like weak associativity), it will \textit{not} be equivalent to Definition \ref{DefMOSVA}. Additional conditions on the order of the pole $z_1=0$ are implicitly assumed in the formal variable formulation. 

\subsection{Pole-order condition}\label{Pole-order}

For simplicity, we discuss only left modules for a MOSVA in this section.  

\begin{defn}\label{NSHTCL}
Let $V$ be a MOSVA. Let $W= \coprod\limits_{m\in \C} W_{[m]}, Y_W^L: V\otimes W \to W[[x, x^{-1}]], \d_W: W \to W$ satisfy axioms for gradings, rationality of products and iterates of two vertex operators and associativity in Definition \ref{DefMOSVA-L} . $Y_W^L$ is said to satisfy the \textit{pole-order condition}, if for every $w'\in W', u_1, u_2\in V, w\in W$, the order of the pole $z_1=0$ of the rational function that$\langle w', Y_W^L(u_1, z_1)Y_W^L(u_2, z_2)w\rangle$ converges to is bounded above by an integer that depends only on $u_1$ and $w$. 
\end{defn}

\begin{rema}\label{PoleCondWeakAssoc}
With the same notations and assumptions in Definition \ref{NSHTCL}, we see that for every $u_1, u_2\in V, w\in W$, $p_1$ appearing in the weak associativity 
$$(x_0+x_2)^{p_1}Y_W^L(u_1, x_0+x_2)Y_W^L(u_2, x_2)w = (x_0+x_2)^{p_1}Y_W^L(Y_V(u_1, x_0)u_2, x_2)w $$
can be chosen as an integer that depends only on $u_1$ and $w$. Conversely, if $W$ and $Y_W^L$ satisfy axioms for gradings, weak associativity with the choice of $p_1$ depending only on $u_1$ and $w$, then one can prove that $Y_W^L$ satisfies the rationality of products and iterates for two vertex operators, associativity and the pole-order condition. 
\end{rema}

\begin{rema}
Note that this condition holds automatically when the commutativity holds. Therefore for vertex algebras, we don't need any extra condition to have a formal variable formulation. 
\end{rema}

\subsection{Rationality of products of any numbers of vertex operators}

The pole-order condition in Definition \ref{NSHTCL} is interesting, because together with some other conditions, it implies the rationality of products of more than two vertex operators. More precisely, we have the following theorem:

\begin{thm}
Let $V$ be a MOSVA with $Y_V$ satisfying the pole-order condition in Definition \ref{NSHTCL}. Let $W = \coprod_{n\in \C} W_{[n]}$, $Y_W^L: V\otimes W \to W[[x, x^{-1}]], \d_W: W \to W, D_W: W \to W$ satisfy the axioms for the grading, the $D$-derivative and $D$-commutator properties, rationality of products and iterates of two vertex operators, associativity, and the pole-order condition in Definition \ref{NSHTCL}. Then rationality of products holds for any numbers of vertex operators. More precisely, for every $u_1, ..., u_n\in V, w'\in W', w\in W$, the series
$$\langle w', Y_W^L(u_1, z_1)\cdots Y_W^L(u_n, z_n)w\rangle$$
converges absolutely when $|z_1|>\cdots > |z_n| >0$ to a rational function with the only possible poles at $z_i=0, i=1, ..., n$ and $z_i=z_j$. Moreover, for each $i=1, ..., n$, the order of the pole $z_i=0$ is bounded above by an integer that depends only on $u_i$ and $w$; for each $i, j$ with $1\leq i < j \leq n$, the order of the pole $z_i=z_j$ is bounded above by an integer that depends only on $u_i$ and $u_j$.  
\end{thm}

\begin{proof}
We first prove the rationality of the product of three vertex operators. Without loss of generality, let $w'\in W',u_1, u_2, u_3\in V, w\in W$ be homogeneous elements. We claim that for some positive integers $p_1, p_2, p_{12}$, 
$$(x_1-x_2)^{p_{12}}(x_1+x_3)^{p_{1}}(x_2+x_3)^{p_2}\langle w', Y_W^L(u_1, x_1+x_3)Y_W^L(u_2, x_2+x_3) Y_W^L(u_3, x_3)w\rangle, $$
as a series in $\C[[x_1, x_1^{-1}, x_2,  x_2^{-1}, x_3,  x_3^{-1}]]$ with all negative powers of $(x_1+x_3)$ and $(x_2+x_3)$ expanded in positive powers of $x_3$, is both upper and lower-truncated. In other words, it is indeed a Laurent polynomial in $\C[x_1, x_1^{-1}, x_2,  x_2^{-1}, x_3,  x_3^{-1}]$. 

We start by peeling off the variable $x_3$. First note the power of $x_3$ is lower-truncated. To see it is also upper-truncated, let $p_{12}$ be an integer such that
$$(x_1-x_2)^{p_{12}} Y_W^L(u_1, x_1)Y_W^L(u_2, x_2)w\in W[[x_1, x_2]][x_1^{-1}, x_2^{-1}]$$
From the discussion in Section \ref{FormFact-1}, we know that $p_{12}$ depends only on $u_1$ and $u_2$. Then we use the weak associativity to find integers $p_1$ and $p_2$, such that 
$$(x_1+x_3)^{p_1}Y_W^L(u_1, x_1+x_3)Y_W^L(u_3, x_3)w = (x_1+x_3)^{p_1} Y_W^L(Y_V(u_1, x_1)u_3, x_3)w$$
as formal series in $W[[x_1, x_1^{-1}, x_3, x_3^{-1}]]$, and 
$$(x_2+x_3)^{p_2}Y_W^L(u_2, x_2+x_3)Y_W^L(u_3, x_3)w = (x_2+x_3)^{p_1} Y_W^L(Y_V(u_2, x_2)u_3, x_3)w$$
as formal series in $W[[x_2, x_2^{-1}, x_3, x_3^{-1}]]$. With these relations, we compute as follows: 
\begin{align}
& (x_1-x_2)^{p_{12}}
(x_1+x_3)^{p_{1}}(x_2+x_3)^{p_2}\langle w', Y_W^L(u_1, x_1+x_3)Y_W^L(u_2, x_2+x_3) Y_W^L(u_3, x_3)w\rangle \label{n-Rat-1}\\
= & (x_1-x_2)^{p_{12}}
(x_1+x_3)^{p_{1}}(x_2+x_3)^{p_2}\langle w', Y_W^L(u_1, x_1+x_3)Y_W^L(Y_V(u_2, x_2)u_3, x_3)w\rangle  \label{n-Rat-2}\\
= & (x_1-x_2)^{p_{12}}
(x_1+x_3)^{p_{1}}(x_2+x_3)^{p_2}\langle w', Y_W^L(Y_V(u_1, x_1)Y_V(u_2, x_2)u_3, x_3)w\rangle \label{n-Rat-3}
\end{align}
where $p_{12}$ is an integers that depends only on $u_1$ and $u_2$.
It is crucial to note that the second equality is guaranteed by the pole-order condition: since $p_1$ depends only on $u_1$ and $w$, we don't need to worry about the infinitely many terms given by $Y_V^L(u_2, x_2)u_3$. If we assume the weaker condition that $p_1$ is independent of only $w'$, the equality might not hold. 

We observe that the series (\ref{n-Rat-3}) is upper-truncated in $x_3$. This can be seen by writing the series (\ref{n-Rat-3}) as 
$$\langle w', Y_W^L(Y_V(u_1, u_1)Y_V(u_2, x_2)u_3, x_3)w = \sum_{m, n}\langle w', Y_W^L(u_{mn}, x_3)w\rangle x_1^{-m-1}x_2^{-n-1}$$
The coefficient of $x_3^{-p-1}$ in $Y_W^L(u_{mn}, x_3)w$ is nonzero only when $\wt u_{mn} + \wt w - p-1 = \wt w'$. Thus 
$$-p-1 = \wt w' - \wt w - \wt u_{mn}$$
As $\wt w'$ and $\wt w$ are fixed, and $\wt u_{mn}$ is bounded below, $-p-1$ is then bounded above. Therefore, the powers of $x_3$ in series (\ref{n-Rat-3}), and thus in series (\ref{n-Rat-1}) and (\ref{n-Rat-2}), are upper-truncated. 

So for the series (\ref{n-Rat-1}), if we expand all the negative powers of $x_1+x_3$ and $x_2+x_3$ as power series in $x_3$ (using $D$-conjugation property from Part (4) of Proposition \ref{ImmediateFacts-L}), 
\begin{align*}
& (x_1-x_2)^{p_{12}}(x_1+x_3)^{p_{1}}(x_2+x_3)^{p_2}\langle w', e^{x_3 D_V}Y_W^L(u_1, x_1)Y_W^L(u_2, x_2)e^{-x_3 D_V} Y_W^L(u_3, x_3)w\rangle\\
= & (x_1-x_2)^{p_{12}}\sum_{k_1=0}^{p_1}\binom{p_1}{k_1} x_1^{p_1-k_1}x_3^{k_1}\sum_{k_2=0}^{p_2}\binom{p_2}{k_2}x_2^{p_2-k_2}x_3^{k_2}\cdot\\
& \qquad \langle w', \sum_{i=1}^\infty \frac 1 {i!}x_3^i D_V^i  Y_W^L(u_1, x_1)Y_W^L(u_2, x_2)\sum_{j=1}^\infty\frac 1 {j!}(-x_3)^j D_V^j \sum_{m\in \Z}(Y_W^L)_m(u_3)w x_3^{-m-1} \rangle
\end{align*}
then since the resulted series has only finitely many positive powers of $x_3$, in all these summations only finitely many terms can survive. And for each fixed $m, p_1, p_2, i, j$, we know that
$$(x_1-x_2)^{p_{12}}\langle w',  D_V^i  Y_W^L(u_1, x_1)Y_W^L(u_2, x_2) D_V^j (Y_W^L)_{m+k_1+k_2+i+j}(u_3)w \rangle$$
is a Laurent polynomial in $x_1, x_2$. So the series is a finite sum of such Laurent polynomials in $x_1, x_2$ multiplied with finitely many powers $x_3$. Thus the claim is proved. 

Since all the powers of $x_1, x_2, x_3$ are lower-truncated, we can find positive integers $p_{13}, p_{23}, p_3$ such that
$$(x_1+x_3)^{p_1} (x_2+x_3)^{p_2} x_3^{p_3} x_1^{p_{13}}x_2^{p_{23}}(x_1-x_2)^{p_{12}}\langle w', Y_W^L(u_1, x_1+x_3)Y_W^L(u_2, x_2+x_3)Y_W^L(u_3, x_3)w\rangle$$
is a polynomial in $\C[x_1, x_2, x_3]$. 
From the lower truncation of the series $Y_W^L(u_3, x_3)w$, we can choose $p_3$ to be an integer depending only on $u_3$ and $w$. From the associativity and the discussion in Remark \ref{FormFact}, we can choose $p_{23}$ to be an integer depending only on $u_2$ and $u_3$. For the integer $p_{1}$, we use associativity to write the above polynomial as 
\begin{align*}
& (x_1+x_3)^{p_1} (x_2+x_3)^{p_2} x_3^{p_3} x_1^{p_{13}}x_2^{p_{23}}(x_1-x_2)^{p_{12}}\langle w', Y_W^L(u_1, x_1+x_3)Y_W^L(Y_V(u_2, x_2)u_3, x_3)w\rangle\rangle\\
=& (x_1+x_3)^{p_1} (x_2+x_3)^{p_2} x_3^{p_3} x_1^{p_{13}}x_2^{p_{23}}(x_1-x_2)^{p_{12}}\sum_{n\in \Z} \langle w', Y_W^L(u_1, x_1)Y_W^L((Y_W^L)_n(u_2)u_3, x_3)w\rangle x_2^{-n-1}
\end{align*}
For each summand, we apply the pole-order condition to find an upper bound of the order of pole of $x_1=0$ that depends only on $u_1$ and $w$. In particular, this upper bound is uniform to all $n$. Thus we see that $p_{13}$ can be chosen as an integer that depends only on $u_1$ and $u_3$. 

Finally, we apply the transformation $x_1 \mapsto x_1-x_3, x_2 \mapsto x_2-x_3, x_3\mapsto x_3$ for the polynomial to see that
$$x_1^{p_1}x_2^{p_2}x_3^{p_3}(x_1-x_2)^{p_{12}}(x_1-x_3)^{p_{13}}(x_2-x_3)^{p_{23}}\langle w', Y_W^L(u_1, x_1)Y_W^L(u_2, x_2)Y_W^L(u_3, x_3)w\rangle$$
is a polynomial in $\C[x_1, x_2, x_3]$, in which the integer $p_i$ depend only on $u_i$ and $w$ for $i=1, 2, 3$, and the integer $p_{ij}$ depend only on $u_i$ and $u_j$ for $1\leq i < j \leq 3$. Thus the series $\langle w', Y_W^L(u_1, z_1)Y_W^L(u_2, z_2)Y_W^L(u_3, z_3)w\rangle$ is the expansion of a rational function with the only possible poles at $z_1=0, z_2=0, z_3=0, z_1=z_2, z_1=z_3, z_2=z_3$ in the region $\{(z_1, z_2, z_3)\in \C^3: |z_1|>|z_2|>|z_3|>0\}$.  

So we proved the rationality of products of three vertex operators based on that of two. For more vertex operators, the rationality can be proved by induction, where the inductive step can easily be modified from the above arguments. 
\end{proof}

\subsection{Formal variable formulation} In regards of Remark \ref{PoleCondWeakAssoc}, we have the following theorem: 

\begin{thm}\label{Formal-left}
Let $V$ be a MOSVA, Let $W = \coprod_{n\in \C} V_{[n]}$, $Y_W^L: V\otimes W \to W[[x, x^{-1}]], \d_W: W\to W$ of weight 0, and $D_W: W\to W$ of weight 1 satisfy axioms for the grading, $D$-derivative property, $D$-commutator formula, and the following \textit{weak associativity with pole-order condition}: for every $u_1, u_2\in V$, $w\in W$, there exists an integer $p_1$ that depends only on $u_1$ and $w$, such that
$$(x_0+x_2)^{p_1}Y_W^L(Y_V(u_1, x_0)u_2, x_2)w = (x_0+x_2)^{p_1} Y_W^L(u_1, x_0+x_2)Y_W^L(u_2, x_2)w$$
as formal series in $W[[x_0, x_0^{-1}, x_2, x_2^{-1}]]$, then 
$(W, Y_W^L, \d_W, D_W)$ forms a left $V$-module, with $Y_W^L$ satisfying the pole-order condition. 
\end{thm}

\begin{rema}
In particular, for the MOSVA itself, its right modules and bimodules, we can formulate similar theorems. Since these are similar, we omit the details.  
\end{rema}

\begin{rema}
We see that that the pole-order condition is crucial for the formal variable approach. But with the absence of commutativity, it might be unnatural to assume. For the other part of the paper, we will still develop the theory of MOSVA using complex analysis. A stronger version of the pole-order condition will be used in Section 6 in the construction of contragredient modules. 
\end{rema}


\section{Rationality of Iterates of Any Numbers of Vertex Operators}

In this section we study the rationality of iterates of any numbers of vertex operators. We will prove that the iterates converge absolutely to the same rational functions as products do. And we will explicitly specify the region of convergence. None of the arguments here rely on the pole-order conditions in Section \ref{Pole-order}. So the conclusions in this section apply to the MOSVAs in the most general sense.

\subsection{A note on complex analysis}

\begin{defn}
A \textit{multicircular domain} $T\subseteq \C^n$ (centered at the origin) is an open subset such that 
$$(z_1, ..., z_n)\in T \text{ implies }(z_1e^{i\theta_1}, ..., z_n e^{i\theta_n})\in T$$
for every $\theta_1, ..., \theta_n \in \R$. The \textit{trace} of a multicircular domain $T\subset \C$ is given by 
$$\tr T = \{(|z_1|, ..., |z_n|)\in \R^n_+: (z_1, ..., z_n)\in T\}$$
\end{defn}

We need the following results in several complex variable functions (see for example \cite{SCV}

\begin{lemma}\label{SVC-Connected}
A multicircular domain $S\subset \C^n$ is connected if and only if $\tr S \subset \R^n_+$ is connected.
\end{lemma}

\begin{proof}
This follows from the continuity of the map $(z_1, ..., z_n)\mapsto (|z_1|, ..., |z_n|)$.
\end{proof}

Given $r=(r_1, ..., r_n)\in \tr T$, we use $\gamma(0, r)$ to denote the multicircle $(r_1e^{i\theta_1}, \cdots r_ne^{i\theta_n})$, where $\theta_i \in \R, i = 1, ..., n$.  

\begin{thm}\label{SVC-ref}
Let $S$ be a connected multicircular domain. 
Let $f$ be a holomorphic function on $S$. Then there is a unique Laurent series in $z_1, ..., z_n$ with center 0 which converges absolutely to $f$ at every point of $S$.
It is the Laurent series 
$$\sum_{\alpha_1\in\Z,...,\alpha_n\in\Z} c_{\alpha_1...\alpha_n}z_1^{-\alpha_1-1}\cdots z_n^{-\alpha_n-1}$$
whose coefficients are given by the formula
$$c_{\alpha_1...\alpha_n} =\frac 1
{(2\pi i)^n}\idotsint_{\gamma (0,r)}f(z_1, ..., z_n)z_1^{\alpha_1}\cdots z_n^{\alpha_n}dz_1 \cdots dz_n$$
for any $r = (r_1, . . . , r_n) > 0$ in the trace of $S$. The series converges uniformly to $f$ on any compact subset of $S$.

\end{thm}

\begin{proof}
See Theorem 1.5.4, Theorem 2.7.1 and the discussion in Section 2.8 of \cite{SCV}, or \cite{LAG}, Section 6.  
\end{proof}

\begin{rema}\label{CoeffDer}
If the Laurent series is lower-truncated in $z_n$, let $-M_n$ be a lower bound of the powers of $z_n$, then one can recover the coefficient of $z_n$ from the derivatives of $z_n^{M_n}f(z_1, ..., z_n)$. More precisely, we have
$$\sum\limits_{\alpha_1,...,\alpha_{n-1} \in \mathbb{Z}} {{c_{\alpha_1...\alpha_n}}} z_1^{\alpha_1} \cdots z_{n-1}^{\alpha_{n-1}} = \frac 1 {(\alpha_n + M_n)!}\mathop {\lim }\limits_{z_n = 0} {\left( {\frac{\partial }{{\partial  z_n}}} \right)^{\alpha_n + M_n}}(z_n^{M_n}f(z_1,...,z_n))$$
\end{rema}



\begin{lemma}\label{IterSeries}
Let $n$ be a positive integer. Let $f$ be a rational function in $z_1, ..., z_n$. Let $T$ be a connected multicircular domain on which the lowest power of $z_n$ in the Laurent series expansion of $f(z_1, ..., z_n)$ is the same as the order of pole $z_n=0$. Let $S$ be a nonempty open subset of $T$ and $S'$ be the image of $S$ via the projection $(z_1, ..., z_n)\mapsto (z_1, ..., z_{n-1})$. Assume that for each fixed $k_n \in \Z$, the series
$$\sum_{k_1, k_2, ..., k_{n-1}\in \Z}a_{k_1k_2...k_{n-1}k_n} z_1^{k_1}z_2^{k_2}\cdots z_{n-1}^{k_{n-1}}$$
converges absolutely for every $(z_1, z_2, ..., z_{n-1})\in S'$, and 
\begin{equation}\label{IterSeries-1}
\sum_{k_n\in \Z}\left(\sum_{k_1, k_2, ..., k_{n-1}\in \Z}a_{k_1k_2...k_{n-1}k_n} z_1^{k_1}z_2^{k_2}\cdots z_{n-1}^{k_{n-1}}\right)z_n^{k_n},
\end{equation}
viewed as a series whose terms are $\left(\sum\limits_{k_1, k_2, ..., k_{n-1}\in \Z}a_{k_1k_2...k_{n-1}k_n} z_1^{k_1}z_2^{k_2}\cdots z_{n-1}^{k_{n-1}}\right)z_n^{k_n}$, is lower-truncated in $z_n$ and converges to $f(z_1, ..., z_n)$ for every $(z_1, z_2, ..., z_{n-1}, z_n)\in S$.  Then the corresponding Laurent series
\begin{equation}\label{IterSeries-2}
\sum_{k_1, k_2, ...,  k_{n-1}, k_n\in \Z}a_{k_1k_2...k_n} z_1^{k_1}z_2^{k_2}\cdots z_{n-1}^{k_{n-1}} z_n^{k_n},
\end{equation}
converges absolutely to $f(z_1, ..., z_n)$ for every $(z_1, ..., z_n)\in T$ 
\end{lemma}


\begin{proof} 
The proof is divided into three parts: 
\begin{enumerate}
\item For any $(z_1,..., z_{n-1}) \in S'$, let $z_n\in \C$ such that $(z_1, ..., z_n)\in S$. Since the series 
$$\sum_{k_n\in \Z}\left(\sum_{k_1, ..., k_{n-1}\in \Z}a_{k_1...k_n} z_1^{k_1}\cdots z_{n-1}^{k_{n-1}}\right)w_n^{k_n},
$$
is lower-truncated in $w_n$ and converges when $w_n = z_n$ and $S$ is open, one can find an integer $M_n$ and $r= |z_n| >0$ such that the series 
$$w_n^{M_n}\sum_{k_n\in \Z}\left(\sum_{k_1, ..., k_{n-1}\in \Z}a_{k_1...k_n} z_1^{k_1}\cdots z_{n-1}^{k_{n-1}}\right)w_n^{k_n},
$$
is a power series in $w_n$ and converges absolutely in the region $\{w_n\in \C: |w_n|<r\}$ to $w_n^{M_n}f(z_1, ..., z_{n-1}, w_n)$. Since the radius of convergence is positive, we can perform term-by-term partial differentiation and evaluate $w_n=0$, to conclude that for each $k_n\in \Z$
$$\sum_{k_1, ..., k_{n-1}\in\Z} a_{k_1...k_n}z_1^{k_1}\cdots z_{n-1}^{k_{n-1}} = \lim_{w_n=0}\left(\frac{\partial}{\partial w_n}\right)^{k_n+M_n} (w_n^{M_2}f(z_1, ..., z_{n-1}, w_n))$$
\item Denote by $g_{k_n}(z_1, ..., z_{n-1})$ the function given by the right-hand-side. Since the series in the left-hand-side converges absolutely in $S'$, one sees easily that it also converges absolutely in the multicircular domain $\{(e^{i\theta_1}z_1, ..., e^{i\theta_{n-1}}z_{n-1}): (z_1, ..., z_{n-1})\in S', \theta_i\in \R, i=1, ...,n-1\}$. Thus $g_{k_n}(z_1, ..., z_{n-1})$ on the right-hand-side is a holomorphic function defined on a connected multicircular domain. With Theorem \ref{SVC-ref}, we know that the Laurent series expansion of $g_{k_n}(z_1, ..., z_{n-1})$ is precisely $\sum\limits_{k_1, k_2, ..., k_{n-1}\in \Z}a_{k_1k_2...k_{n-1}k_n} z_1^{-k_1-1}z_2^{-k_2-1}\cdots z_{n-1}^{-k_{n-1}-1}$ and thus
$$a_{k_1k_2...k_{n-1}k_n} = \idotsint_\gamma f(z_1, ..., z_n)dz_1\cdots dz_n. $$
for any multicircle $\gamma$ that lies in the multicircular domain. 
\item Now we consider the Laurent expansion of the function $f(z_1, ..., z_n)$ in the region $T$. By Theorem \ref{SVC-ref}, this function can be expanded as a unique Laurent series in $z_1, ..., z_n$. Moreover by assumption, the lowest power of $z_n$ in this series is bounded below by the negative of the order of the pole $z_n=0$, which is easily seen to be less negative than $M_n$. Therefore with Theorem \ref{SVC-ref} and Remark \ref{CoeffDer}, one can see that for each $k_1, ..., k_n\in \Z$, the coefficient of $z_1^{k_1}\cdots z_n^{k_n}$ in this expansion is 
$$\idotsint_\gamma z_1^{-k_1-1}\cdots z_{n-1}^{-k_{n-1}} \lim_{z_n=0}\left(\frac{\partial}{\partial z_n}\right)^{k_n+M_n} (z_n^{M_n}f(z_1, ..., z_n))dz_1\cdots dz_{n-1}$$
which coincides with $a_{k_1...k_n}$. As the expansion is done in the larger multicircular domain $T$, we proved that the series (\ref{IterSeries-2}) converges absolutely for every $(z_1, ..., z_n)\in T$. 
\end{enumerate}
\end{proof}

\begin{rema}
Note that the series (\ref{IterSeries-1}) is an iterated series. So the theorem concludes the convergence of the series (\ref{IterSeries-2}) from the convergence of the iterated series (\ref{IterSeries-1}) to a holomorphic function. 
\end{rema}

For iterated series that are upper truncated, we also have a similar result. 

\begin{lemma}\label{IterSeries-Inf}
Let $n$ be a positive integer. Let $f$ be a rational function in $z_1, ..., z_n$. Let $T$ be a connected multicircular domain on which the highest power of $z_n$ in the Laurent series expansion of $f(z_1, ..., z_n)$ is the same as the negative of the order of pole $z_n=\infty$. Let $S$ be a nonempty open subset of $T$ and $S'$ be the image of $S$ via the projection $(z_1, ..., z_n)\mapsto (z_1, ..., z_{n-1})$. Assume that for each fixed $k_n \in \Z$, the series
$$\sum_{k_1, k_2, ..., k_{n-1}\in \Z}a_{k_1k_2...k_{n-1}k_n} z_1^{k_1}z_2^{k_2}\cdots z_{n-1}^{k_{n-1}}$$
converges absolutely for every $(z_1, z_2, ..., z_{n-1})\in S'$, and 
\begin{equation*}
\sum_{k_n\in \Z}\left(\sum_{k_1, k_2, ..., k_{n-1}\in \Z}a_{k_1k_2...k_{n-1}k_n} z_1^{k_1}z_2^{k_2}\cdots z_{n-1}^{k_{n-1}}\right)z_n^{k_n},
\end{equation*}
viewed as a series whose terms are $\left(\sum\limits_{k_1, k_2, ..., k_{n-1}\in \Z}a_{k_1k_2...k_{n-1}k_n} z_1^{k_1}z_2^{k_2}\cdots z_{n-1}^{k_{n-1}}\right)z_n^{k_n}$, is upper-truncated in $z_n$ and converges to $f(z_1, ..., z_n)$ for every $(z_1, z_2, ..., z_{n-1}, z_n)\in S$.  Then the corresponding Laurent series
\begin{equation*}
\sum_{k_1, k_2, ...,  k_{n-1}, k_n\in \Z}a_{k_1k_2...k_n} z_1^{k_1}z_2^{k_2}\cdots z_{n-1}^{k_{n-1}} z_n^{k_n},
\end{equation*}
converges absolutely to $f(z_1, ..., z_n)$ for every $(z_1, ..., z_n)\in T$ 
\end{lemma}

\begin{proof}
It suffices to perform the transformation $z_n \mapsto 1/z_n$ and apply the Lemma \ref{IterSeries}
\end{proof}


\subsection{Rationality of iterates of any numbers of vertex operators}\label{Section-Iterate}

With the above preparations, we now deal with the iterates of any numbers of vertex operators. For simplicity, we first investigate the iterate of three vertex operators, i.e. 
$$\langle w', Y_W^L(Y_V(Y_V(u_1, z_1-z_2)u_2, z_2-z_3)u_3, z_3)w\rangle.$$
To show that this series converges absolutely and to find the region of convergence, we need the following intermediate proposition:

\begin{prop}\label{label-1} 
For any $u_1, u_2, u_3\in V, w\in W, w'\in w'$, fixed $z_1, z_2, z_3\in \C$ satisfying $|z_2|>|z_1-z_2-z_3|, |z_2|>|z_1-z_2|>0, |z_2|>|z_3|>0$, the series
$$\langle w',  Y_W^L(Y_V(u_1, z_1-z_2)u_2, z_2-z_3) Y_W^L(u_3, z_3)w\rangle$$
converges absolutely to the rational function that 
$$\langle w', Y_W^L(u_1, z_1)Y_W^L(u_2, z_2)Y_W^L(u_3,z_3)w\rangle$$
converges to. 
\end{prop}

\begin{proof} 
From the rationality of products, we know that
\begin{align*}
&\langle w', Y_W^L(u_1, z_1)Y_W^L(u_2, z_2)Y_W^L(u_3,z_3)w\rangle \\
& \qquad = \sum_{k_1, k_2, k_3\in \Z} \langle w', (Y_W^L)_{k_1}(u_1) (Y_W^L)_{k_2}(u_2) (Y_W^L)_{k_3}(u_3)w \rangle z_1^{-k_1-1}z_2^{-k_2-1}z_3^{-k_3-1}
\end{align*}
converges absolutely in the region $\{(z_1, z_2, z_3)\in \C^3: |z_1|>|z_2|>|z_3|>0\}$ to a rational function that has the only possible poles at $z_1=0, z_2=0, z_3=0, z_1=z_2, z_2=z_3, z_1=z_3$. Denote this rational function by $f(z_1, z_2, z_3)$. Then 
\begin{equation}\label{3-iter-corr}
f(z_1, z_2, z_3) = \frac{g(z_1, z_2, z_3)}{z_1^{p_1}z_2^{p_2}z_3^{p_3}(z_1-z_2)^{p_{12}}(z_2-z_3)^{p_{23}}(z_1-z_3)^{p_{13}}}
\end{equation}
for some integers $p_1, p_2,p_3, p_{12},p_{23},p_{13} \geq 0$ and some polynomial $g(z_1, z_2, z_3)$. 

Now for each fixed $k_3\in \Z$, we consider the series
\begin{align*}
&\langle w', Y_W^L(u_1, z_1)Y_W^L(u_2, z_2)(Y_W^L)_{k_3}(u_3)w\rangle \\
& \qquad = \sum_{k_1, k_2\in \Z} \langle w', (Y_W^L)_{k_1}(u_1) (Y_W^L)_{k_2}(u_2) (Y_W^L)_{k_3}(u_3)w \rangle z_1^{-k_1-1}z_2^{-k_2-1}.
\end{align*}
As part of an absolutely convergent series, it is also absolutely convergent. From associativity, its sum is equal to 
$$\langle w', Y_W^L(Y_V(u_1, z_1-z_2)u_2, z_2) (Y_W^L)_{k_3}(u_3)w\rangle$$
when $|z_1|>|z_2|>|z_1-z_2|>0$ for each fixed $l$. We multiply it with $z_3^{-k_3-1}$ and sum up all $k_3\in \Z$ to see that 
\begin{align*}
& \sum_{k_3\in \Z}\langle w', Y_W^L(Y_V(u_1, z_1-z_2)u_2, z_2) (Y_W^L)_{k_3}(u_3)w\rangle z_3^{-k_3-1}\\
& \quad = \sum_{k_3\in \Z}\left(\sum_{k_1, k_2\in \Z} \langle w', (Y_W^L)_{k_2}((Y_V)_{k_1}(u_1)u_2) (Y_W^L)_{k_3}(u_3)w \rangle (z_1-z_2)^{-k_1-1}z_2^{-k_2-1}\right)z_3^{-k_3-1}
\end{align*}
viewed as a single complex series whose terms are 
$$\left(\sum_{k_1, k_2\in \Z} \langle w', (Y_W^L)_{k_2}((Y_V)_{k_1}(u_1)u_2) (Y_W^L)_{k_3}(u_3)w \rangle (z_1-z_2)^{-k_1-1}z_2^{-k_2-1}\right)z_3^{-k_3-1},$$
converges to $f(z_1, z_2, z_3)$ when $|z_1|>|z_2|>|z_3|>0, |z_2|>|z_1-z_2|>0$. Moreover, one checks easily that the power of $z_3$ is lower-truncated. 

We now use Lemma 
\ref{IterSeries} to elaborately show that the series
\begin{align*}
& \langle w', Y_W^L(Y_V(u_1, z_1-z_2)u_2, z_2) Y_W^L(u_3, z_3)w\rangle \\
& \qquad = \sum_{k_1, k_2, k_3\in \Z} \langle w', (Y_W^L)_{k_2}((Y_V)_{k_1}(u_1)u_2) (Y_W^L)_{k_3}(u_3)w \rangle (z_1-z_2)^{-k_1-1}z_2^{-k_2-1}z_3^{-k_3-1}
\end{align*}
converges absolutely to $f(z_1, z_2, z_3)$ when $|z_2|>|z_1-z_2-z_3|, |z_2|>|z_1-z_2|>0, |z_2|>|z_3|>0$. First we set $\zeta_1=z_1-z_2, \zeta_2=z_2, \zeta_3=z_3$. Let
$$T = \{(\zeta_1, \zeta_2, \zeta_3): |\zeta_2|>|\zeta_3|+|\zeta_1|, |\zeta_1|>0, |\zeta_3|>0\}$$
With Lemma \ref{SVC-Connected}, we see that $T$ is a connected multicircular domain. Now we express the function $f(z_1, z_2, z_3)$ in terms of the variables $\zeta_1, \zeta_2, \zeta_3$ as
$$f(\zeta_1+\zeta_2, \zeta_2, \zeta_3) = \frac{g(\zeta_1+\zeta_2, \zeta_2, \zeta_3)}{(\zeta_1+\zeta_2)^{p_1}\zeta_2^{p_2}\zeta_3^{p_3}\zeta_1^{p_{12}}(\zeta_2-\zeta_3)^{p_{23}}(\zeta_1 + \zeta_2 -\zeta_3)^{p_{13}}},$$
which admits a Laurent series expansion in $\zeta_1, \zeta_2, \zeta_3$ by the following steps: 
\begin{enumerate}
\item Expand the negative powers of $\zeta_1 + \zeta_2$ as a power series in $\zeta_1$. The resulted series converges when $|\zeta_2|>|\zeta_1|$. 
\item Expand the negative powers of $\zeta_2 - \zeta_3$ as a power series in $\zeta_3$. The resulted series converges when $|\zeta_2|>|\zeta_3|$.
\item Expand the negative powers of $\zeta_1 + \zeta_2 - \zeta_3$ as power series in $\zeta_1 - \zeta_3$, then further expand all the positive power of $\zeta_1 - \zeta_3$ as polynomials. The resulted series converges in $|\zeta_2|>|\zeta_1 - \zeta_3|$. 
\end{enumerate}
Obviously if $(\zeta_1, \zeta_2,\zeta_3)\in T$, then all the above conditions are satisfied (note that $|\zeta_2|>|\zeta_3|+|\zeta_1|$ implies that $|\zeta_2|>|\zeta_1-\zeta_3|$ by triangle inequality). Thus $f(\zeta_1+\zeta_2, \zeta_2, \zeta_3)$ is expanded as an absolutely convergent Laurent series in $T$. From Theorem \ref{SVC-ref}, the Laurent series is unique. Note that the lowest power of $\zeta_3$ in this Laurent is $-p_3$.

Set 
$$S = \{(\zeta_1, \zeta_2, \zeta_3)\in \C^3:|\zeta_1+\zeta_2|>|\zeta_2|>|\zeta_3|>0, |\zeta_2|>|\zeta_1|>0 \} \cap T$$
Obviously, $S$ is a nonempty open subset of $T$. We know that the series 
$$\sum_{k_1, k_2\in \Z} \langle w', (Y_W^L)_{k_2}((Y_V)_{k_1}(u_1)u_2) (Y_W^L)_{k_3}(u_3)w \rangle \zeta_1^{-k_1-1}\zeta_2^{-k_2-1}\zeta_3^{-k_3-1}  $$
is absolutely convergent whenever $(\zeta_1, \zeta_2, \zeta_3)\in S$, and the series
$$\sum_{k_3\in \Z}\left(\sum_{k_1, k_2\in \Z} \langle w', (Y_W^L)_{k_2}((Y_V)_{k_1}(u_1)u_2) (Y_W^L)_{k_3}(u_3)w \rangle \zeta_1^{-k_1-1}\zeta_2^{-k_2-1}\right)\zeta_3^{-k_3-1}, $$
viewed as a series whose terms are $\sum\limits_{k\in \Z}\langle w', Y_W^L(\pi_k^V Y_V(u_1, \zeta_1)u_2, \zeta_2) \pi_l^W Y_W^L(u_3, \zeta_3)w\rangle $, is lower-truncated in $\zeta_3$ and absolutely convergent to $f(\zeta_1+\zeta_2, \zeta_2, \zeta_3)$ whenever $(\zeta_1, \zeta_2, \zeta_3)\in S$. Thus Lemma \ref{IterSeries} implies that the series 
$$\sum_{k_3\in \Z}\sum_{k_1, k_2\in \Z} \langle w', (Y_W^L)_{k_2}((Y_V)_{k_1}(u_1)u_2) (Y_W^L)_{k_3}(u_3)w \rangle \zeta_1^{-k_1-1}\zeta_2^{-k_2-1}\zeta_3^{-k_3-1} $$
converges absolutely when $(\zeta_1, \zeta_2, \zeta_3)\in T$. Finally, since the expansion of the rational function is given in the region  
$$\{(\zeta_1, \zeta_2, \zeta_3)\in \C^3: |\zeta_2|>|\zeta_1-\zeta_3|, |\zeta_2|>|\zeta_1|>0, |\zeta_2|>|\zeta_3|>0\},$$
the Laurent series also converges absolutely in this region. 
That is to say, in terms of variables $z_1, z_2, z_3$, the series 
$$\sum_{k_1, k_2, k_3\in \Z} \langle w', (Y_W^L)_{k_2}((Y_V)_{k_1}(u_1)u_2) (Y_W^L)_{k_3}(u_3)w \rangle (z_1-z_2)^{-k_1-1}z_2^{-k_2-1}z_3^{-k_3-1}$$
converges absolutely to $f(z_1, z_2, z_3)$ when $|z_2|>|z_1-z_2-z_3|, |z_2|>|z_1-z_2|>0,|z_2|>|z_3|>0$. 
\end{proof}

\begin{thm}
For any $u_1, u_2, u_3\in V, w\in W, w'\in w'$, fixed $z_1, z_2, z_3\in \C$ satisfying $|z_3|>|z_1-z_3|, |z_3|>|z_2-z_3|>|z_1-z_2|>0$, the series
$$\langle w', Y_W^L(Y_V(Y_V(u_1, z_1-z_2)u_2, z_2-z_3)u_3, z_3)w\rangle$$
converges absolutely to the rational function that 
$$\langle w', Y_W^L(u_1, z_1)Y_W^L(u_2, z_2)Y_W^L(u_3,z_3)w\rangle$$
converges to. 
\end{thm}

\begin{proof}
We proceed similarly based on the result above: in the series
\begin{align*}
& \langle w', Y_W^L(Y_V(u_1, z_1-z_2)u_2, z_2)Y_W^L(u_3, z_3)w\rangle \\
& \qquad = \sum_{k_1, k_2, k_3\in \Z} \langle w', (Y_W^L)_{k_2}((Y_V)_{k_1}(u_1)u_2) (Y_W^L)_{k_3}(u_3)w \rangle (z_1-z_2)^{-k_1-1}z_2^{-k_2-1}z_3^{-k_3-1}
\end{align*}
we fix $k_1$ and consider the series 
\begin{align*}
& \langle w', Y_W^L((Y_V)_{k_1}(u_1)u_2, z_2)Y_W^L(u_3, z_3)w\rangle \\
& \qquad = \sum_{k_2, k_3\in \C} \langle w', (Y_W^L)_{k_2}((Y_V)_{k_1}(u_1)u_2) (Y_W^L)_{k_3}(u_3)w \rangle z_2^{-k_2-1}z_3^{-k_3-1}. 
\end{align*}
As part of an absolutely convergent series, this series is also absolutely convergent. From associativity, its sum is equal to 
\begin{align*}
& \langle w', Y_W^L(Y_V((Y_V)_{k_1}(u_1)u_2, z_2-z_3)u_3, z_3)w\rangle \\
& \qquad = \sum_{k_2, k_3\in \Z} \langle w', (Y_W^L)_{k_3}((Y_V)_{k_2}((Y_V)_{k_1}(u_1)u_2)u_3)w\rangle (z_2-z_3)^{-k_2-1} z_3^{-k_3-1}. 
\end{align*}
when $|z_2|>|z_3|>|z_2-z_3|>0$. We multiply it with $(z_1-z_2)^{-k_1-1}$ sum up all $k_1\in \Z$. With the conclusion of the previous proposition, we see that 
\begin{align*}
& \sum_{k_1\in \Z}\langle w', Y_W^L(Y_V((Y_V)_{k_1}(u_1)u_2, z_2-z_3)u_3, z_3)w\rangle (z_1-z_2)^{-k_1-1}\\
& \quad = \sum_{k_1\in \Z}\left(\sum_{k_2, k_3\in \Z} \langle w', (Y_W^L)_{k_3}((Y_V)_{k_2}((Y_V)_{k_1}(u_1)u_2)u_3)w\rangle (z_2-z_3)^{-k_2-1} z_3^{-k_3-1}\right) (z_1-z_2)^{-k_1-1}. 
\end{align*}
viewed as a single complex series whose terms are 
$$\left(\sum_{k_2, k_3\in \Z} \langle w', (Y_W^L)_{k_3}((Y_V)_{k_2}((Y_V)_{k_1}(u_1)u_2)u_3)w\rangle (z_2-z_3)^{-k_2-1} z_3^{-k_3-1}\right) (z_1-z_2)^{-k_1-1}, $$
converges absolutely to $f(z_1, z_2, z_3)$ when $|z_2|>|z_2-z_3|, |z_2|>|z_3|+|z_1-z_2|, |z_1-z_2|>0, |z_3|>0$, for the same $f(z_1, z_2, z_3)$ as that in Formula (\ref{3-iter-corr}). Moreover, one sees that the power of $(z_1-z_2)$ in this series is lower-truncated. 

We similarly use Lemma 
\ref{IterSeries} to elaborately show that the series
\begin{align*}
& \langle w', Y_W^L(Y_V(Y_V(u_1, z_1-z_2)u_2, z_2-z_3)u_3, z_3)w\rangle \\
& \quad = \sum_{k_1, k_2, k_3\in \Z} \langle w', (Y_W^L)_{k_3}((Y_V)_{k_2}((Y_V)_{k_1}(u_1)u_2)u_3)w\rangle (z_1-z_2)^{-k_1-1}(z_2-z_3)^{-k_2-1} z_3^{-k_3-1} . 
\end{align*}
converges absolutely to $f(z_1, z_2, z_3)$ when $|z_3|>|z_1-z_3|, |z_2-z_3|>|z_1-z_2|>0$. 

First we perform the transformation $\zeta_1= z_1-z_2, \zeta_2=z_2-z_3, \zeta_3= z_3$. Set  
$$T = \{(\zeta_1, \zeta_2, \zeta_3): |\zeta_3| > |\zeta_1|+|\zeta_2|, |\zeta_2| > |\zeta_1|>0\}$$
With Lemma \ref{SVC-Connected}, we see that $T$ is a connected multicircular domain. Moreover, $T$ is a subset of $\{(\zeta_1, \zeta_2, \zeta_3)\in \C^3: |\zeta_i|>|\zeta_1|, i = 2, 3\}$.
We express the function $f(z_1, z_2, z_3)$ in terms of the variables $\zeta_1, \zeta_2, \zeta_3$ as
$$f(\zeta_1+\zeta_2+\zeta_3, \zeta_2+\zeta_3, \zeta_3) = \frac{g(\zeta_1+\zeta_2+\zeta_3, \zeta_2+\zeta_3, \zeta_3)}{(\zeta_1+\zeta_2+\zeta_3)^{p_1}(\zeta_2+\zeta_3)^{p_2}\zeta_3^{p_3}\zeta_1^{p_{12}}\zeta_2^{p_{23}}(\zeta_1 + \zeta_2)^{p^{13}}},$$
which admits a Laurent series expansion in the following steps:
\begin{enumerate}
\item Expand negative powers of $\zeta_1+\zeta_2+\zeta_3$ as power series in $\zeta_1+\zeta_2$, then further expand the positive powers of $\zeta_1+\zeta_2$ as polynomials in $\zeta_1$ and $\zeta_2$. This series converges absolutely when $|\zeta_3|>|\zeta_1+\zeta_2|$
\item Expand negative powers of $\zeta_2+\zeta_3$ as power series in $\zeta_2$. This series converges absolutely when $|\zeta_3|>|\zeta_2|$
\item Expand negative powers of $\zeta_1+\zeta_2$ as power series in $\zeta_1$. This series converges absolutely when $|\zeta_2|>|\zeta_1|$
\end{enumerate}
Obviously if $(\zeta_1, \zeta_2, \zeta_3)\in T$, then all the above conditions are satisfied (Note that $|\zeta_3|>|\zeta_1|+|\zeta_2|$ implies that $|\zeta_3|>|\zeta_1+\zeta_2|$ by triangle inequality). Thus $f(\zeta_1+\zeta_2, \zeta_2+\zeta_3, \zeta_3)$ is expressed as an absolutely convergent Laurent series. From Theorem \ref{SVC-ref}, the Laurent series is unique. 

Set 
$$S = \{(\zeta_1, \zeta_2, \zeta_3): |\zeta_2|>|\zeta_3|+|\zeta_1|, |\zeta_1|>0, |\zeta_3|>0\}\cap T. $$
So $S$ is a nonempty open subset of $T$. We know that the series 
\begin{align*}
\sum_{k_2, k_3\in \Z} \langle w', (Y_W^L)_{k_3}((Y_V)_{k_2}((Y_V)_{k_1}(u_1)u_2)u_3)w\rangle \zeta_2^{-k_2-1} \zeta_3^{-k_3-1} \zeta_1^{-k_1-1}. 
\end{align*}
converges absolutely when $(\zeta_1, \zeta_2, \zeta_3)\in S$, and the series
$$\sum_{k_1\in \Z}\left(\sum_{k_2, k_3\in \Z} \langle w', (Y_W^L)_{k_3}((Y_V)_{k_2}((Y_V)_{k_1}(u_1)u_2)u_3)w\rangle \zeta_2^{-k_2-1} \zeta_3^{-k_3-1}\right) \zeta_1^{-k_1-1}, $$
viewed as a series whose terms are 
$$\sum_{k_2, k_3\in \Z} \langle w', (Y_W^L)_{k_3}((Y_V)_{k_2}((Y_V)_{k_1}(u_1)u_2)u_3)w\rangle \zeta_2^{-k_2-1} \zeta_3^{-k_3-1} \zeta_1^{-k_1-1}, $$ 
converges absolutely to $f(\zeta_1+\zeta_2, \zeta_2+\zeta_3, \zeta_3)$ when $(\zeta_1, \zeta_2, \zeta_3)\in S$. Thus Lemma \ref{IterSeries} implies that the triple series
$$\sum_{k_1\in \Z}\sum_{k_2, k_3\in \Z} \langle w', (Y_W^L)_{k_3}((Y_V)_{k_2}((Y_V)_{k_1}(u_1)u_2)u_3)w\rangle \zeta_2^{-k_2-1} \zeta_3^{-k_3-1} \zeta_1^{-k_1-1},$$
converges absolutely when $(\zeta_1, \zeta_2, \zeta_3)\in T$. Finally, as the expansion is done in the region
$$\{(\zeta_1, \zeta_2, \zeta_3)\in \C^3: |\zeta_3|>|\zeta_1+\zeta_2|>0, |\zeta_3|>|\zeta_2|>|\zeta_1|>0\}$$
the series also converges absolutely in this region. 
That is to say, in terms of variables $z_1, z_2, z_3$, the series
$$\sum_{k_1, k_2, k_3\in \Z} \langle w', (Y_W^L)_{k_3}((Y_V)_{k_2}((Y_V)_{k_1}(u_1)u_2)u_3)w\rangle (z_2-z_3)^{-k_2-1} z_3^{-k_3-1} (z_1-z_2)^{-k_1-1}$$
converges absolutely to $f(z_1, z_2, z_3)$ when $|z_3|>|z_1-z_3|, |z_3|>|z_2-z_3|>|z_1-z_2|>0$. 
\end{proof}

One can similarly prove the following theorem regarding the rationality of iterates of $n$ vertex operators. 

\begin{thm}\label{n-iter-prop}
For $u_1, u_2, ..., u_n\in V, w\in W, w'\in W'$, the series
$$\langle w', Y_W^L(Y_V(\cdots Y_V(Y_V(u_1, z_1-z_2)u_2, z_2-z_3)u_3 \cdots, z_{n-1}-z_n )u_n, z_n)w \rangle$$
converges absolutely in the region 
\begin{equation}\label{IterRegion}
\left\{(z_1, ..., z_n)\in \C^n: \begin{aligned}
&|z_n|>|z_i-z_n|>0, i = 1, ..., n; \\
&|z_i-z_{i+1}|>|z_j-z_i|>0, 1\leq j < i \leq n-1 \end{aligned}\right\}
\end{equation}
to the rational function that 
$$\langle w', Y_W^L(u_1, z_1)Y_W^L(u_2, z_2)\cdots Y_W^L(u_n, z_n)w\rangle$$ 
converges to. 
\end{thm}

\begin{proof}
One can use induction. The inductive step can be carried out using the analytic continuation with Lemma 
\ref{IterSeries} as shown in the above argument. Details are technical and thus are omitted here. See \cite{Thesis} for details
\end{proof}


\begin{rema}
Similar results hold for right modules and bimodules. We shall state the results here. The proof is almost the identical. 
\end{rema}


\begin{thm}\label{n-iter-prop-R}
For $u_1, u_2, ..., u_n\in V, w\in W, w'\in W'$, the series
$$\langle w', Y_W^R(Y_W^R(\cdots Y_W^R(Y_W^R(w, z_1-z_2)u_1, z_2-z_3)u_2 \cdots, z_{n-1}-z_n )u_{n-1}, z_n)u_n \rangle$$
converges absolutely in the region (\ref{IterRegion})
to the rational function that 
$$\langle w', Y_W^R(w, z_1)Y_V(u_1, z_2)\cdots Y_V(u_{n-1}, z_n)u_n\rangle$$ 
converges to. 
\end{thm}

\begin{thm}
For $u_1, ..., u_n, u_{n+1}, ... u_{n+m}\in V, w\in W, w'\in W'$, the series
$$\langle w', Y_W^R( \cdots Y_W^R(Y_W^L({Y_V}( \cdots ({Y_V}({u_1},{z_1} - {z_2}) \cdots {u_n},{z_n} - {z_{n + 1}})w,{z_{n + 1}} - {z_{n + 2}}) \cdots {u_{n + m - 1}},{z_{n + m}}){u_{n + m}}\rangle $$
converges absolutely in the region (\ref{IterRegion})
to the rational function determined by 
$$\langle w', Y_W^L(u_1, z_1) \cdots Y_W^L(u_n, z_n) Y_W^R(w, z_{n+1}) Y_V(u_{n+1}, z_{n+2})\cdots Y_V(u_{n+m-1}, z_{n+m})u_{n+m}\rangle$$
\end{thm}

\begin{rema}
For mixtures of products and iterates, we can use the same argument to determine their regions of convergence. 
\end{rema}

\section{Opposite MOSVAs and Modules}

In this section, for a given MOSVA, we will introduce its opposite MOSVA using the skew-symmetry opposite vertex operator. This can be viewed as the analogue of the opposite associative algebra of a given associative algebra. Analogously, we prove that a right module for the MOSVA is the same as the left module for the opposite MOSVA, and a left module for the MOSVA is the same as the right module for the opposite MOSVA. The rationality of iterates we proved in the Section \ref{Section-Iterate} will be used in the proofs of these theorems. 

\subsection{The skew-symmetry opposite vertex operator} \label{Skew-Symm-Section}

\begin{defn}
Let $(V, Y_V, \one)$ be a MOSVA. The \textit{skew-symmetry opposite vertex operator} $Y_V^s$ of $Y_V$ is defined as follows 
$$\begin{aligned}
&Y_V^s: & V\otimes V &\to V[[x, x^{-1}]]\\
&& u\otimes v &\mapsto e^{xD_V} Y_V(v, -x)u
\end{aligned}$$
\end{defn}

One sees easily that the series defining the skew-symmetry opposite vertex operator is lower truncated. 

\begin{prop}
For $v'\in V', u_1, u_2, v\in V$, the double complex series 
$$\langle v', Y_V^s(u_1, z_1)Y_V^s(u_2, z_2)v\rangle $$
converges absolutely when $|z_1|>|z_2|>0$ to a rational function with the only possible poles $z_1=0, z_2=0, z_1=z_2$.
\end{prop}

\begin{proof}
The proof will be divided in three steps. 
\begin{enumerate}
\item From the conclusion of Theorem \ref{n-iter-prop} that the series
$$\langle v', Y_V(Y_V(v, -z_2)u_2, -z_1+\zeta_2)u_1\rangle$$
converges absolutely when $|-z_1+\zeta_2|>|z_2|>0$ to a rational function that has the only possible poles at $z_2=0, z_1=\zeta_2, z_1+z_2 = \zeta_2$, with the $D$-conjugation property (See Part (4) of Proposition \ref{ImmediateFacts}) and Lemma \ref{IterSeries}, we can prove that the series
$$\langle v', Y_V(e^{\zeta_2 D_V} Y_V(v, -z_2)u_2, -z_1)u_1\rangle$$
converges absolutely when $|z_1|>|z_2-\zeta_2|, |z_1|>|\zeta_2|, |z_2|>0$ to a rational function that has the only possible poles at $z_1+z_2=\zeta_2, z_2=0, z_1=\zeta_2$. The argument is very similar to that in the proof of Theorem \ref{n-iter-prop} and is omitted here. 
\item Since $\zeta_2= z_2$ is contained in the region of the convergence, we then evaluate $\zeta_2=z_2$ to see that the series 
$$\langle v', Y_V(e^{z_2 D_V} Y_V(v, -z_2)u_2, -z_1)u_1\rangle$$
converges absolutely when $|z_1|>|z_2|>0$ to the rational function determined by 
$$\langle v', Y_V(Y_V(v, -z_2)u_2, -z_1+z_2)u_1\rangle$$
that has the only possible poles at $z_1=0, z_2=0, z_1=z_2$. 
\item Now we argue that for every $v'\in V', u_1, u_2, v\in V$, the series
\begin{equation}\label{ezD}
\langle v', e^{z_1D_V}Y_V(e^{z_2D_V}Y_V(v, -z_2)u_2, -z_1)u_1\rangle
\end{equation}
converges absolutely in the same region $S$. We first note that the adjoint $D_V': V^* \to V^*$ of $D_V$, defined by 
$$\langle D_V' v', v\rangle = \langle v', D_V v\rangle, v'\in V', v\in V, $$
restrics to a homogeneous linear operator on $V'$ of weight $-1$. Thus for every $z\in \C^\times$, the action of $e^{zD_V'}$ on $v'\in V'$ is a finite sum of elements of $V'$. So the series (\ref{ezD}) is the same as 
$$\langle e^{z_1D_V'}v', Y_V(e^{z_2D_V}Y_V(v, -z_2)u_2, -z_1)u_1\rangle$$
which is a finite sum of series that converges absolutely to rational functions with the only possible poles at $z_1=0, z_2=0, z_1=z_2$. Thus the sum also converges absolutely in the same region to a rational function of the same type. 
\end{enumerate}
\end{proof}

\begin{rema}\label{expD}
The argument we have used in dealing with $e^{zD}$ operator can be generalized to products and iterates of any numbers of vertex operators. One should also note that we don't need $V$ to be grading-restricted. The same result also holds for left modules, right modules and bimodules for MOSVAs. For brevity, in the future we will not repeat the argument, but refer to this remark when we need the $e^{zD}$ operator. 
\end{rema}

\begin{prop} \label{n-prod-opp}
For $v'\in V', u_1, ..., u_n, v\in V$, the complex series 
$$\langle v', Y_V^s(u_1, z_1)\cdots Y_V^s(u_n, z_n)v\rangle $$
converges absolutely when $|z_1|>\cdots > |z_n|>0$ to a rational function with the only possible poles at $z_i=0, i=1,..., n; z_i = z_j, 1\leq i < j \leq n. $
\end{prop}

\begin{proof}
Likewise, the proof is divided into three steps. For brevity, we only state the conclusions of each step. 
\begin{enumerate}
\item With the conclusion of Theorem \ref{n-iter-prop}, the $D$-conjugation property (See Part (4) of Proposition \ref{ImmediateFacts}) and Lemma \ref{IterSeries} we can prove that the series
$$\langle v', Y_V(e^{\zeta_2D_V}\cdots Y_V(e^{\zeta_{n-1}D_V}Y_V(e^{\zeta_n D_V}Y_V(v, -z_n)u_n,-z_{n-1})u_{n-1},-z_{n-2})\cdots, -z_1)u_1\rangle$$
converges absolutely when 
\begin{align*}
|z_k|&>|\zeta_{k+1}+ (-z_{k+1} + \zeta_{k+2}) + \cdots + (-z_{n-1}+\zeta_n)-z_n|, k = 1, ..., n-1,\\
|z_k|&>|\zeta_{k+1}+ (-z_{k+1} + \zeta_{k+2}) + \cdots + (-z_i + \zeta_{i+1})|, k = 1, ..., n-1, i = k, ..., n-1.
\end{align*}
to the rational function determined by 
$$\langle v', Y_V(\cdots Y_V(Y_V(Y_V(v, -z_n)u_n,-z_{n-1}+\zeta_n)u_{n-1},-z_{n-2}+\zeta_{n-1})\cdots, -z_1+\zeta_2)u_1\rangle$$
that has the only possible poles at 
\begin{align*}
-z_n + (z_{n-1}+\zeta_n) + \cdots + (-z_k + \zeta_{k+1}) &= 0, k = 1, ..., n-1; \\
(-z_i + \zeta_{i+1}) + \cdots + (-z_k + \zeta_{k+1}) &= 0, k = 1, ..., n-1, i = k, ..., n-1. 
\end{align*}
\item Since $\zeta_i= z_i, i = 2, ..., n$ is contained in the region of the convergence, we then evaluate $\zeta_i=z_i$ for every $i=2, ..., n$, to see that the series 
$$\langle v', Y_V(e^{z_2D_V}\cdots Y_V(e^{z_{n-1}D_V}Y_V(e^{z_n D_V}Y_V(v, -z_n)u_n,-z_{n-1})u_{n-1},-z_{n-2})\cdots, -z_1)u_1\rangle$$
converges absolutely when $|z_1|>\cdots > |z_n|>0$ to the rational function determined by 
$$\langle v', Y_V(\cdots Y_V(Y_V(Y_V(v, -z_n)u_n,-z_{n-1}+z_n)u_{n-1},-z_{n-2}+z_{n-1})\cdots, -z_1+z_2)u_1\rangle$$
that has the only possible poles at $z_i=0, i = 1, ..., n; z_i=z_j, 1\leq i < j \leq n$. 
\item Finally we use Remark \ref{expD} to conclude that the series
$$\langle v', Y_V^s(u_1, z_1)\cdots Y_V^s(u_n, z_n)v\rangle, $$
which is precisely 
$$\langle v', e^{z_1D_V}Y_V(e^{z_2D_V}\cdots Y_V(e^{z_{n-1}D_V}Y_V(e^{z_n D_V}Y_V(v, -z_n)u_n,-z_{n-1})u_{n-1},-z_{n-2})\cdots, -z_1)u_1\rangle,$$
converges absolutely when $|z_1|>\cdots > |z_n|>0$ to a rational function that has the same types of poles. 
\end{enumerate}
\end{proof}

\begin{prop} \label{2-iter-opp}
For $v'\in V', u_1, u_2, v\in V$, the complex double series
$$\langle v', Y_V^s(Y_V^s(u_1,z_1-z_2)u_2, z_2)v\rangle$$
converges absolutely when $|z_2|>|z_1-z_2|>0$ to a rational function with the only possible poles at $z_1=0, z_2=0, z_1=z_2$
\end{prop}

\begin{proof}
\begin{enumerate}
\item With the rationality of products the $D$-conjugation property and Lemma \ref{IterSeries}, we can prove that the series
$$\langle v', e^{\zeta D_V}Y_V(v, -z_2)e^{-\zeta D_V}Y_V(u_2, -z_1+z_2)u_1\rangle$$
converges absolutely when $|z_2|>|\zeta|, |z_2|>|z_1-z_2+\zeta|, |z_1-z_2|>0$
to the rational function determined by 
$$\langle v', Y_V(v, -z_2+\zeta) Y_V(u_2, -z_1+z_2)u_1\rangle$$
that has the only possible poles at $z_2=\zeta, z_1=z_2, z_1-z_2+\zeta=z_2$
\item Since $\zeta= -z_1+z_2$ is contained in the region of the convergence, we then evaluate $\zeta=-z_1+z_2$ to see that the series 
$$\langle v', e^{(-z_1+z_2) D_V}Y_V(v, -z_2)e^{(z_1-z_2) D_V}Y_V(u_2, -z_1+z_2)u_1\rangle$$
converges absolutely when $|z_2|>  |z_1-z_2|>0$ to the rational function determined by 
$$\langle v', Y_V(v, -z_1) Y_V(u_2, -z_1+z_2)u_1\rangle$$
that has the only possible poles at $z_1=0, z_1=z_2, z_2 = 0$. 
\item Finally we use Remark \ref{expD} to conclude that the series
$$\langle v', Y_V^s(Y_V^s(u_1, z_1-z_2)u_2, z_2)v\rangle, $$
which is precisely 
$$\langle v', e^{z_2D_V}Y_V(v, -z_2)e^{(z_1-z_2)D_V}Y_V(u_2, -z_1+z_2)u_1\rangle,$$
converges absolutely when $|z_2|>|z_1-z_2|>0$ to a rational function that has the same types of poles. 
\end{enumerate}
\end{proof}



\subsection{The opposite MOSVA}

\begin{thm}\label{oppMOSVA}

Given a MOSVA $(V, Y_V, \one)$, with the skew-symmetry vertex operator map
$$\begin{aligned}
&Y_V^s: & V\otimes V &\to V[[x, x^{-1}]]\\
&& u\otimes v &\mapsto e^{xD_V} Y_V(v, -x)u. 
\end{aligned}$$
$(V, Y_V^s, \one)$ is also a MOSVA. 
\end{thm}

\begin{proof}
\begin{enumerate}
\item The lower bound condition is trivial. We verify the $\d_V$-commutator formula: for every $u\in V$
$$[\d_V, Y_V^s(u,x)] = x\frac d {dx}Y_V^s(u,x) + Y_V^s(\d_V u, x). $$
First we note that since $D_V$ is of weight 1, for each $n\in \mathbb{N}$, $D_V^n$ is then of weight $n$. Thus
$$[\d_V, D_V^n] = nD_V^n,$$
and one computes easily that
$$[\d_V, e^{xD_V}] = xD_Ve^{xD_V}. $$
With this formula, we compute the commutator as follows
\begin{align*}
[\d_V, Y_V^s(u,x)]v & = \d_V Y_V^s(u, x)v - Y_V^s(u, x)\d_V v\\
&= \d_V e^{xD_V}Y_V(v, -x)u - e^{xD_V}Y_V(\d_V v, -x)u\\
&= e^{xD_V} \d_V Y_V(v, -x)u + xD_Ve^{xD_V}Y_V(v, -x)u - e^{xD_V}Y_V(\d_V v, -x)u \\
&= e^{xD_V} Y_V(v, -x)\d_Vu + e^{xD_V}Y_V(\d_V v, -x)u + e^{xD_V} x\frac d {dx}Y_V(v, -x)u \\
& \qquad + xD_Ve^{xD_V}Y_V(v, -x)u - e^{xD_V}Y_V(\d_V v, -x)u \\
&= e^{xD_V} Y_V(v, -x)\d_Vu  + e^{xD_V} x\frac d {dx}Y_V(v, -x)u  + xD_Ve^{xD_V}Y_V(v, -x)u \\
&= e^{xD_V} Y_V(v, -x)\d_Vu + x\frac d {dx} e^{xD_V}Y_V(v, -x)u\\
&= Y_V^s(\d_V u, x)v + x \frac d {dx} Y_V^s(u, x)v,
\end{align*}
where the fourth equality holds because of the $\d$-commutator formula for $Y_V$; and the sixth equality holds because of the formal chain rule.

\item Since for $v\in V$, 
$$Y_V^s(\one, x)v = e^{xD_V} Y_V(v, -x)\one = e^{xD_V} e^{-xD_V} v = v, $$
the identity property follows. Since for $u\in V$, 
$$Y_V^s(u,x)\one = e^{zD_V}Y_V(\one, -x) u = e^{zD_V} u,$$
the creation property follows. 

\item It follows directly from $Y_V^s(u,x)\one = e^{xD_V}u$ that
$$D_V u = \lim\limits_{x\to 0} \frac d{dx} Y_V^s(u, x)\one. $$
We prove the $D$-derivative property as follows:
$$\begin{aligned}
  Y_V^s({D_V}u,x)v &= {e^{x{D_V}}}{Y_V}(v, - x){D_V}u = {e^{x{D_V}}}{D_V}{Y_V}(v, - x)u + {e^{x{D_V}}}[{D_V},{Y_V}(v, - x)]u \hfill \\
   &= {e^{x{D_V}}}{D_V}{Y_V}(v, - x)u + {e^{x{D_V}}}\frac{d}{{d( - x)}}{Y_V}(v, - x)u \hfill \\
   &= \frac{d}{{dx}}{e^{x{D_V}}}{Y_V}(v, - x)u = \frac{d}{{dx}}Y_V^s(u,x)v. \hfill \\ 
\end{aligned}$$
Then the $D_V$-commutator formula follows from
$$\begin{aligned}[]
[D_V, Y_V^s(u,x)]v &= D_VY_V^s(u,x)v - Y_V^s(u,x)D_Vv \\ 
&= D_V e^{xD_V}Y(v,-x)u - e^{xD_V}Y(D_Vv, -x)u\\
&= {e^{x{D_V}}}{D_V}{Y_V}(v, - x)u + {e^{x{D_V}}}\frac{d}{{d( - x)}}{Y_V}(v, - x)u = Y^o(D_Vu, x)V. \hfill \\\end{aligned}$$

\item This has been done in Proposition \ref{n-prod-opp} and Proposition \ref{2-iter-opp}. 

\item Fix $u_1, u_2, v\in V$ and $v'\in V'$. Let $S_1 = \{(z_1, z_2): |z_1|>|z_2|>0\}$ and $S_2= \{(z_1, z_2): |z_2|>|z_1-z_2|>0\}$. A careful analysis of the proof to Proposition \ref{n-prod-opp} shows that, the series 
$$\langle v', Y_V^s(u_1, z_1)Y_V^s(u_2, z_2)v\rangle$$
converges absolutely in $S_1$ to the same rational function as what 
$$\langle v', e^{z_1D_V}Y_V(Y_V(v, -z_2)u_2, -z_1+z_2)u_1\rangle$$
converges to (when $|z_1-z_2|>|z_2|>0$). Also the proof to Proposition \ref{2-iter-opp} shows that when $|z_2|>|z_1-z_2|>0$, the series 
$$\langle v', Y_V^s(Y_V^s(u_1, z_1-z_2)u_2, z_2)v\rangle$$
converges absolutely in $S_2$ to the same rational function as that
$$\langle v', e^{z_1D_V}Y_V(v, -z_1)Y_V(u_2, -z_1+z_2)u_1\rangle$$
converges to (when $|z_1|>|z_1-z_2|>0$). From the associativity of $Y_V$, we know that these rational functions are identical. In other words, $\langle v', Y_V^s(u_1, z_1)Y_V^s(u_2, z_2)v\rangle$ and $\langle v', Y_V^s(Y_V^s(u_1, z_1-z_2)u_2, z_2)v\rangle$ converges absolutely to the same rational function respectively in the region $S_1$ and $S_2$. So in $S_1\cap S_2$ their sums are equal. 
\end{enumerate}

\end{proof}

The MOSVA $(V, Y_V^s, \one)$ is called the \textit{opposite MOSVA} of $(V, Y_V, \one)$. When there is no confusion, we will use $V^{op}$ to denote it.

\begin{rema}
Given a MOSVA $(V, Y_V, \one)$, from the fact that
$$(Y_V^s)^s(u,x)v = e^{xD_V}Y_V^s(v,-x)u = e^{xD_V}e^{-xD_V}Y_V(u,x)v = Y_V(u,x)v, $$
we have $(V^{op})^{op} = V$. 
\end{rema}

\begin{rema}
For a vertex algebra with a lower-bounded grading, we know that $Y_V = Y_V^s$ because this is precisely the skew-symmetry identity. Conversely, 
if a MOSVA $V$ satisfies $Y_V = Y_V^s$, i.e.  for $v'\in V', u_1, u_2, v\in V$ and any $x\neq 0$,
$$\langle v', Y_V(u, z)v\rangle = \langle v', e^{zD_V}Y_V(v, -z)u\rangle, $$
then $V$ is a vertex algebra with a lower-bounded grading, since associativity and skew-symmetry identity imply the Jacobi identity (see \cite{H-Gen-Rat} Proposition 2.2 and \cite{LL}, Section 3.6). 
\end{rema}


\subsection{Modules for the MOSVA and modules for the opposite MOSVA}

\begin{thm}\label{leftVopModule}
Given a right $V$-module $(W, Y_W^R, \d_W, D_W)$, we define the vertex operator map
$$\begin{aligned}
Y_W^{s(R)}: & V \otimes W \to W\\
& v \otimes w \mapsto e^{xD_W}Y_W^R(w, -x)v. 
\end{aligned}$$
Then $(W, Y_W^{s(R)}, \d_W, D_W)$ is a left $V^{op}$-module. \\
Conversely, given a left $V^{op}$-module $(W, Y_W^{s(R)}, \d_W, D_W)$, we define the vertex operator map
$$\begin{aligned}
Y_W^R: &W\otimes V \to W\\
& w \otimes v \mapsto e^{xD_W}Y_W^{s(R)}(v, -x)w. 
\end{aligned}$$
then $(W, Y_W^R, \d_W, D_W)$ is a right $V$-module. 
\end{thm}

\begin{proof}
Let $(W, Y_W^R, \d_W, D_W)$ be a right $V$-module. We verify all the axioms of the left $V^{op}$-module. 
\begin{enumerate}
\item The grading of $W$ obviously satisfy the lower bound condition and the $\d$-grading condition. The proof of the $\d$-commutator formula is similar to that in the proof of Theorem \ref{oppMOSVA}.  
\item The identity property follows from Proposition \ref{ImmediateFacts-R}
$$Y_W^{s(R)}(\one, x)w = e^{xD_W}Y_W^R(w, -x)\one = e^{xD_W}e^{-xD_W}w  = w. $$
\item We first prove the $D$-derivative property
$$\begin{aligned}
\frac{d}{{dx}}Y_W^{s(R)}(v,x)w &= \frac{d}{{dx}}\left( {{e^{x{D_W}}}{Y_W^R}(w, - x)v} \right) = {D_W}{e^{x{D_W}}}{Y_W^R}(w, - x)v + {e^{x{D_W}}}\frac{d}{{dx}}\left( {{Y_W^R}(w, - x)v} \right)\\
&= {e^{x{D_W}}}{D_W}{Y_W^R}(w, - x)v + {e^{x{D_W}}}\frac{d}{{dx}}\left( {{Y_W^R}(w, - x)v} \right) \hfill \\
&= {e^{x{D_W}}}[{D_W},{Y_W^R}(w, - x)]v + {e^{x{D_W}}}{Y_W^R}(w, - x){D_V}v + {e^{x{D_W}}}\frac{d}{{dx}}\left( {{Y_W^R}(w, - x)v} \right) \hfill \\
&= {e^{x{D_W}}}\frac{d}{{d( - x)}}{Y_W^R}(w, - x)v + {e^{x{D_W}}}{Y_W^R}(w, - x){D_V}v + {e^{x{D_W}}}\frac{d}{{dx}}\left( {{Y_W^R}(w, - x)v} \right) \hfill \\
&= {e^{x{D_W}}}{Y_W^R}(w, - x){D_V}v = Y_W^{s(R)}({D_V}v,x)w.  
\end{aligned}$$
The $D$-commutator formula follows: 
$$\begin{aligned}[]
  [{D_W},Y_W^{s(R)}(v,x)]w &= {D_W}{e^{x{D_W}}}{Y_W^R}(w, - x)v - {e^{x{D_W}}}{Y_W^R}({D_W}w, - x)v \hfill \\
   &= {e^{x{D_W}}}{D_W}{Y_W^R}(w, - x)v + {e^{x{D_W}}}\frac{d}{{dx}}\left( {{Y_W^R}(w, - x)v} \right)\\ 
   &= \frac{d}{{dx}}\left( {{e^{x{D_W}}}{Y_W^R}(w, - x)v}\right)= \frac{d}{{dx}}Y_W^{s(R)}(v,x)w.  \hfill
\end{aligned}
$$
\item It suffices to replace $Y_V$ by $Y_W^R$ and $Y_V^s$ by $Y_W^{s(R)}$ in the arguments of Proposition \ref{n-prod-opp} and Proposition \ref{2-iter-opp}.  

\item It suffices to replace $Y_V$ by $Y_W^R$ and $Y_V^s$ by $Y_W^{s(R)}$ in the arguments of Part (5) of Theorem \ref{oppMOSVA}. 

\end{enumerate}
The converse can be proved similarly. We omit the details here.
\end{proof}

Similarly, one can prove the following theorem:

\begin{thm}
Given a left $V$-module $(W, Y_W^L, \d_W, D_W)$, we define the vertex operator map
$$\begin{aligned}
Y_W^{s(L)}: & W \otimes V \to W\\
& w \otimes v \mapsto e^{xD_W}Y_W^L(v, -x)w.
\end{aligned}$$
Then $(W, Y_W^{s(L)}, \d_W, D_W)$ is a left $V^{op}$-module. \\
Conversely, given a right $V^{op}$-module $(W, Y_W^{s(L)}, \d_W, D_W)$, we define the vertex operator map
$$\begin{aligned}
Y_W^L: &V\otimes W \to W\\
& v \otimes w \mapsto e^{xD_W}Y_W^{s(L)}(w, -x)v.
\end{aligned}$$
Then $(W, Y_W^L, \d_W, D_W)$ is a left $V$-module. 
\end{thm}


\section{M\"obius Structure and Contragredient Modules}

In this section we define M\"obius structure on MOSVAs and the left (right, bi-) modules for such MOSVAs. With this structure, we prove that the graded dual of a grading-restricted  left module for a MOSVA forms a M\"obius right module for the MOSVA. For M\"obius left modules that are not grading-restricted, we prove the same result under a pole-order condition stronger than that in Definition \ref{NSHTCL}. The results in this section generalize the theory of contragredient modules for M\"obius vertex algebras developed in \cite{FHL} and \cite{HLZ1}.

\subsection{Basic definitions}

\begin{defn}
A \textit{M\"obius MOSVA} is a MOSVA $(V, Y_V, \one)$ with a representation $\rho_V$ of the Lie algebra $\mathfrak{sl}(2)$ on $V$, given by 
$$L_V(0) = \rho_V(L_0) = \d_V, L_V(-1) = \rho_V(L_1) = D_V, L_V(1) = \rho_V(L_1)$$
where $\{L_{-1}, L_0, L_1\}$ is a basis of $\mathfrak{sl}(2)$ with Lie commutators
$$[L_0, L_{-1}]=L_{-1}, [L_0, L_1]=-L_1, \text{ and }[L_{-1}, L_1] = -2L_0,$$
and the following conditions hold for every $u\in V$:
$$[L_V(1), Y_V(u, x)]=Y(L_V(1)u,x) + 2xY(L_V(0)u, x)+x^2Y(L_V(-1)u,x)$$
We will use the notation $(V, Y_V, \one, \rho_V)$ to denote a M\"obius MOSVA. When there is no confusion, we will simply use the notation $V$. 
\end{defn}

\begin{rema}
Since $\d_V = L_V(0)$ and $[L_V(0), L_V(1)] = -L_V(1)$, we know that $L_V(1)$ is actually a linear operator of weight $-1$. Since the grading on $V$ is lower-bounded, the operator is actually locally nilpotent, i.e., for every $v\in V$, there exists $m\in \Z_+$ such that $L_V(1)^m v = 0$. Moreover, with the identity property and creation property, we can see that 
$$L_V(j) \one = 0, j=0, \pm 1$$
\end{rema}

\begin{prop}\label{Opp-Mob-MOSVA}
Let $(V, Y_V, \one, \rho_V)$ be a M\"obius MOSVA. Then the opposite MOSVA $(V, Y_V^s, \one, \rho_V)$ is also a M\"obius MOSVA.
\end{prop}

\begin{proof}
It suffices to check the commutator formula
$$[L_V(1),Y_V^s(u,x)] = Y_V^s(L_V(1)u,x) + 2xY_V^s(L_V(0)u,x) + {x^2}Y_V^s(L_V(-1)u,x).$$
We first compute the left-hand-side:
\begin{align*}
[L_V(1),Y_V^s(u,x)]v & = L_V(1)Y_V^s(u,x)v - Y_V^s(u,x)L_V(1)v\\
& = L_V(1){e^{xL_V(-1)}}Y_V(v, - x)u - {e^{xL_V(-1)}}Y_V(L_V(1)v, - x)u
\end{align*}
In order to interchange $L_V(1)$ and $e^{xL_V(-1)}$ that appear in the first term, we note that for every $n\in \N$,
$$L_V(1)L_V(-1)^n = L_V(-1)^nL_V(1) + L_V(-1)^{n-1}2nL_V(0) +  n(n-1) L_V(-1)^{n-1}, $$
which can be easily proved by induction. Then a straightforward computation shows that
$$L_V(1){e^{xL_V(-1)}} =  e^{xL_V(-1)}L_V(1) + 2xe^{xL_V(-1)}L_V(0) + x^2e^{xL_V( - 1)}L_V(-1). $$
So the left-hand-side is 
\begin{align*}
& e^{xL_V(-1)}L_V(1)Y_V(v,-x)u + 2xe^{xL_V(-1)}L_V(0)Y_V(v,-x)u \\
& \qquad + {x^2}e^{xL_V(-1)}L_V(-1)Y_V(v,-x)u - e^{xL_V(-1)}Y(L_V(1)v, - x)u
\end{align*}
Then we use the commutator relation between $L_V(j), j=0, \pm 1$ and $Y_V(v, -x)$ to deal with the first three terms. The first term is equal to 
\begin{align*}
& e^{xL_V(-1)}Y_V(v, -x)L_V(1)u + e^{xL_V(-1)}Y_V(L_V(1)v, - x)u \\ 
& \qquad - 2xe^{xL_V(-1)}Y_V(L_V(0)v, -x)u + {x^2}e^{xL_V(-1)}Y_V(L_V(-1)v, -x)u
\end{align*}
The second term is equal to 
$$2xe^{xL_V(-1)}Y_V(v, -x)L_V(0)u + 2xe^{xL_V(-1)}Y_V(L_V(0)v, -x)u - 2{x^2}e^{xL_V(-1)}Y_V(L_V(-1)v, -x)u$$
The third term is equal to
$$x^2{e^{xL_V(-1)}}Y_V(v, -x)L_V(-1)u + x^2{e^{xL_V(-1)}}Y_V(L_V(-1)v, -x)u$$
The summation of the above three formulas, together with the fourth term, would then simplify to the right-hand-side. 
\end{proof}

\begin{defn}
Let $(V, Y_V, \one, \rho_V)$ be a M\"obius MOSVA. A \textit{M\"obius left $V$-module $W$} is a left $V$-module $(W, Y_W^L, \d_W, D_W)$ with a representation $\rho_W$ of the Lie algebra $\mathfrak{sl}(2)$ on $W$, such that 
$$L_W(0) = \rho_W(L_0), L_W(-1)= \rho_W(L_{-1}) = D_W, L_W(1) = \rho_W(L_1),$$
and for every $u\in V$, 
$$[L_W(0), Y_W^L(u,x)] = Y_W^L(L_V(0)u, x) + x Y_W^L(L_V(-1)u, x)$$
$$[L_W(1), Y_W^L(u,x)] = Y_W^L(L_V(1)u, x) + 2x Y_W^L(L_V(0)u, x) + x^2 Y_W^L(L_V(-1)u, x), $$
and for every $n\in \C, w\in W_{[n]}$, there exists $m\in \mathbb{N}$ such that $(L_W(0)-n)^m w=0$. 
\end{defn}

We will use the notation $(W, Y_W^L, \rho_W)$ to denote M\"obius left $V$-modules. The operator $\d_W$ can be defined as the semisimple part of $L_W(0)$, and the operator $D_W$ is just $L_W(-1)$. So the representation $\rho_W$ has all the information of these two operators and thus we don't need to include them in the notation. When there is no confusion, we will simply use $W$. 

\begin{rema}\label{d-operator}
In \cite{HLZ1}, modules in which $L_W(0)$ is not semisimple are called generalized modules. In the MOSVA setting, we don't use this terminology because we are not requiring the operator $\d_W$ to coincide with $L_W(0)$. Indeed, given $L_W(0)$ satisfying the commutator formulas, one can define $\d_W$ as the semisimple of $L_W(0)$. By similar arguments as those in \cite{HLZ1}, we have 
$$[\d_W, (Y_W^L)_n(v)] = [L_W(0), (Y_W^L)_n(v)] \text{ for all }v \in V \text{ and } n \in Z; $$
$$[\d_W, L_W(j)] = [L_W(0), L_W(j)] \text{ for }j = 0, \pm 1.$$
Thus a M\"obius left $V$-module is still a left $V$-module and should not be entitled with the word ``generalized''. 
\end{rema} 

\begin{rema}
In accordance with convention, when we discuss MOSVA and modules with M\"obius structure, we will refer $\d$-commutator formula as $L(0)$-commutator formula, $D$-derivative property and $D$-commutator formula as $L(-1)$-commutator formula. 
\end{rema}

\begin{defn}
Let $(V, Y_V, \one, \rho_V)$ be a M\"obius MOSVA. A \textit{M\"obius right $V$-module $W$} is a right $V$-module $(W, Y_W^R, \d_W, D_W)$ with a representation $\rho_W$ of the Lie algebra $\mathfrak{sl}(2)$ on $W$, such that 
$$L_W(0) = \rho_W(L_0), L_W(-1)= \rho_W(L_{-1}) = D_W, L_W(1) = \rho_W(L_1),$$
and for every $w\in W$, 
$$[L_W(0), Y_W^R(u,x)] = Y_W^L(L_W(0)w, x) + x Y_W^L(L_W(-1)w, x)$$
$$L_W(1)Y_W^R(w,x)-Y_W^R(w,x)L_V(1) = Y_W^L(L_W(1)w, x) + 2x Y_W^L(L_W(0)w, x) + x^2 Y_W^L(L_W(-1)w, x), $$
and for every $n\in \C, w\in W_{[n]}$, there exists $m\in \mathbb{N}$ such that $(L_W(0)-n)^m w=0$
\end{defn}

\begin{rema}\label{leftMobVopModule}
With similar arguments as Proposition \ref{Opp-Mob-MOSVA}, one can prove the M\"obius version of Theorem \ref{leftVopModule}. In particular, $(W, Y_W^R, \rho_W)$ is a M\"obius right $V$-module if and only if $(W, Y_W^{s(R)}, \rho_W)$ is a M\"obius left $V^{op}$-module, where $Y_W^{s(R)}$ and $Y_W^R$ are skew-symmetry opposite vertex operators to each other. This will be used in the proof of Theorem \ref{W'Module} and \ref{W'Module-1}. 
\end{rema}

\begin{defn}
Let $(V, Y_V, \one, \rho_V)$ be a M\"obius MOSVA. A \textit{M\"obius $V$-bimodule} $W$ is a $V$-bimodule $(W, Y_W^L, Y_W^R, \d_W, D_W)$ with a representation $\rho_W$ of the Lie algebra $\mathfrak{sl}(2)$ on $W$, such that $(W, Y_W^L, \rho_W)$ forms a M\"obius left $V$-module, and $(W, Y_W^R, \rho_W)$ forms a M\"obius right $V$-module. 
\end{defn}

\subsection{The opposite vertex operator}

\begin{defn}
Let $(V, Y_V, \one, \rho_V)$ be a M\"obius MOSVA and $(W, Y_W^L, \rho_W)$ be a M\"obius left $V$-module. We define the \textit{opposite vertex operator} on $W$ associated to $u\in V$ by
$$Y_W^o(u, x) = Y_W^L(e^{xL(1)} (-x^{-2})^{L(0)}u, x^{-1}).$$
For homogeneous $u\in V$, we have
\begin{align*}
Y_W^o(u, x) &= \sum_{n\in \Z} (Y_W^o)_n(u) x^{-n-1}\\
&= \sum_{n\in \Z} \left((-1)^{\text{wt }u} \sum_{m=0}^\infty \frac 1 {m!} (Y_W^L)_{-n-m-2+2\text{wt }u}(L(1)^m v) \right)x^{-n-1}
\end{align*}
Note that since $L(1)$ is locally nilpotent, the summation about variable $m$ is actually finite. Thus each component $(Y_W^o)_n(u)$ is well-defined. Also, the order of summation can be switched at our convenience.
\end{defn}

\begin{rema}
The opposite vertex operator we are defining here should not be confused with the skew-symmetry opposite vertex operator we introduced in the Section \ref{Skew-Symm-Section}. In case $V$ is a vertex algebra, $W$ is a $V$-module, the skew-symmetry opposite vertex operator $Y_W^{s(L)}$ coincides with the original vertex operator $Y_W^L$, while the opposite vertex operator $Y_W^o$ is different. See \cite{FHL} and \cite{HLZ1} for details.  
\end{rema}

\begin{prop}\label{YWO-rat}
For every $u_1, ..., u_n\in V, w\in W, w'\in W'$, the series
$$\langle w', Y_W^o(u_n, z_n)\cdots Y_W^o(u_1, z_1)w\rangle$$
converges absolutely when $|z_1|>\cdots > |z_n|>0$ to a rational function with the only possible poles at
$z_i=0, i=1,..., n$ and $z_i = z_j, 1\leq i < j \leq n$. 
\end{prop}

\begin{proof}
It suffices to consider the case when $u_1, ..., u_n\in V$ are homogeneous. In this case,
\begin{align*}
& \langle w', Y_W^o(u_n, z_n)\cdots Y_W^o(u_1, z_1)w\rangle \\ 
= & \sum_{m_1, ..., m_n \text{ finite}} (-1)^{\text{wt }u_1 + \cdots + \text{wt }u_n} z_1^{-2\text{wt }u_1}\cdots z_n^{-2 \text{wt }u_n}\langle w', Y_W^L(L(1)^{m_n}u)_n, z_n^{-1})w\cdots Y_W^L(L(1)^{m_1}u_1, z_1^{-1})\rangle. 
\end{align*}
By the rationality of $Y_W^L$, for fixed $m_1, ..., m_n$, $\langle w', Y_W^L(L(1)^{m_n}u)_n, z_n^{-1})w\cdots Y_W^L(L(1)^{m_1}u_1, z_1^{-1})\rangle$ converges absolutely when $|z_n^{-1}|>\cdots > |z_1|^{-1}> 0$ to a rational function of the form
$$\frac{f(z_1^{-1}, ..., z_n^{-1})}{\prod\limits_{i=1}^n z_i^{-p_i}\prod\limits_{1\leq i < j \leq n} (z_i^{-1}-z_j^{-1})^{p_{ij}}}= \frac{f(z_1^{-1}, ..., z_n^{-1})\prod\limits_{i=1}^n z_i^{p_i+\sum\limits_{j=i+1}^n p_{ij}}}{\prod\limits_{1\leq i < j \leq n} (z_j-z_i)^{p_{ij}}}$$
As the polyonomial $f(z_1^{-1}, ..., z_n^{-n})$ provides negative powers of $z_i, i=1,...,n$, this fraction is a rational function with possible poles at $z_i=0, i=1, ..., n$ and $z_i = z_j, 1\leq i < j \leq n$. Then $\langle w', Y_W^o(u_n, z_n)\cdots Y_W^o(u_1, z_1)w\rangle$, as a finite sum of absolutely convergent series, also converges absolutely when $|z_1|>\cdots > |z_n|>0$ to a rational function with the only possible poles at $z_i=0, i = 1,...,n$ and $z_i = z_j, 1\leq i< j \leq n$. 
\end{proof}

\begin{prop}\label{YWO-ass}
For every $u_1, u_2\in V, w\in W, w'\in W'$. the series
$$\langle w', Y_W^o(Y_V(u_2, z_2-z_1)u_1, z_1)w\rangle$$
converges absolutely when $|z_1|>|z_2-z_1|>0$ to a rational function with the only possible poles at $z_1=0, z_2=0, z_1=z_2$. Moreover,  
$$\langle w', Y_W^o(u_2, z_2)Y_W^o(u_1, z_1)w\rangle = \langle w', Y_W^o(Y_V(u_2, z_2-z_1)u_1, z_1)w\rangle$$
when $|z_1|>|z_2|>|z_1-z_2|>0$
\end{prop}

\begin{proof}
We will use Formula (5.2.35) in \cite{FHL}: for every $u\in V$, we have 
$$e^{xL(1)}(-x^{-2})^{L(0)}Y_V(u, x_0) = Y_V\left(e^{(x+x_0)L(1)}(-(x+x_0)^{-2})^{L(0)}u, -\frac{x_0}{(x+x_0)x}\right) e^{xL(1)}(-x^{-2})^{L(0)}$$
as formal series in $(\text{End }V)[[x, x^{-1}, x_0, x_0^{-1}]]$ where all the negative powers of $x+x_0$ are expanded as power series in $x_0$. The proof of the formula can be found in \cite{FHL}, Section 5.2. The idea is to use the $L(0)$-commutator formula and $L(1)$-commutator formula to obtain $L(0)$-conjugation formula and $L(1)$-conjugation formula. No other property was needed. So the proof carries over to MOSVAs and their modules. 

To apply this formula, we first study the formal series
$$\langle w', Y_W^L\left({e^{{x_1}L(1)}}{( - x_1^{ - 2})^{L(0)}}{Y_V}({u_2},x_0){u_1},x_1^{ - 1}\right)w\rangle$$
in $\C[[x_0, x_0^{-1}, x_1, x_1^{-1}]]$.
By the formula above, the formal series is equal to 
\begin{equation}\label{YWO-ass-2}\left\langle w', Y_W^L\left(Y_V\left(e^{(x_1+x_0)L(1)}(-(x_1+x_0)^{-2})^{L(0)}u, -\frac{x_0}{(x_1+x_0)x_1}\right) e^{x_1L(1)}(-x_1^{-2})^{L(0)}{u_1},x_1^{ - 1}\right)w\right\rangle
\end{equation}
in $\C[[x_0, x_0^{-1}, x_1, x_1^{-1}]]$, with all the negative powers of $x_1+x_0$ expanded as power series in $x_0$. Moreover, it is easy to see that this series has at most finitely many negative powers of $x_0$ and at most finitely many positive powers of $x_1$. 

In order to substitute $x_0$ and $x_1$ by complex numbers $z_0$ and $z_1$, we first note from the rationality of iterates of two vertex operators, for complex numbers $z_0, z_1, \zeta_0, \zeta_1$ with $|\zeta_1
|> |z_0\zeta_1/((z_1+\zeta_0))|>0$, i.e., $|z_1+\zeta_0|>|z_0|>0, |\zeta_1|>0$, the complex series 
\begin{align*}
& \left\langle w', Y_W^L\left(Y_V\left(e^{(z_1+\zeta_0)L(1)}(-(z_1+\zeta_0)^{-2})^{L(0)}u, -\frac{z_0\zeta_1}{(z_1+\zeta_0)}\right) e^{z_1L(1)}(-z_1^{-2})^{L(0)}{u_1},\zeta_1\right)w\right\rangle \\
& \qquad = \sum_{i \text{ finite}}\sum_{m,n} a_{mni}(z_1+\zeta_0)^i\left(-\frac{z_0\zeta_1}{(z_1+\zeta_0)}\right)^{-m-1}(\zeta_1)^{-n-1}
\end{align*}
(with variables $-z_0\zeta_1/((z_1+\zeta_0))$ and $\zeta_1$) converges absolutely to a rational function with the only possible poles at $z_0=0, z_1=0, \zeta_1=0, z_1+\zeta_0=0, z_1+\zeta_0=z_0$ (note the operators $e^{(z_1+\zeta_0)L(1)}$ and $e^{z_1L(1)}$ acts as polynomials, and $(-(z_1+\zeta_0)^{-2})^{-L(0)}$ and $(-z_1^{-2})^{L(0)}$ acts by a scalar multiplication on homogeneous elements). Note that in this expansion, the power of $(-z_0\zeta_1/(z_1+\zeta_0))$ is lower-truncated, i.e., $m$ is bounded above. In particular, the power of $z_0$ is lowert-truncated. Moreover, for each fixed $m$, the power of $\zeta_1$ is lower-truncated, i.e., $n$ is also bounded above for each fixed $m$.

Now, we further expand the negative powers of $z_1+\zeta_0$ as power series in $\zeta_0$, i.e., 
\begin{align*}
& \sum_{i \text{ finite}}\sum_{m,n} a_{mni}(z_1+\zeta_0)^i\left(-\frac{z_0\zeta_1}{(z_1+\zeta_0)}\right)^{-m-1}\zeta_1^{-n-1}\\
& \qquad= \sum_{i \text{ finite}}\sum_{m,n} a_{mni}(-1)^{m+1}z_0^{-m-1}\zeta_1^{-m-n-2}\left(\sum_{k=0}^\infty \binom{m+1+i}{k} z_1^{m+1+i-k} \zeta_0^k\right). 
\end{align*}
The resulting iterated series on the right-hand-side converges absolutely to the rational function when $|z_1+\zeta_0|>|z_0|>0, |z_1|>|\zeta_0|, |\zeta_1|>0$. We check that all the conditions of Lemma \ref{IterSeries} is satisfied. Thus the complex series corresponding to the iterated series on the right-hand-side is precisely the Laurent series expansion of the rational function when $|z_1|>|\zeta_0|, |z_1|>|\zeta_0-z_0|, |\zeta_1|>0, |z_0|>0$. In particular, the complex series converges absolutely when $|z_1|>|\zeta_0|, |z_1|>|\zeta_0-z_0|, |\zeta_1|>0, |z_0|>0$. Now we substitute $\zeta_0 = z_0, \zeta_1=z_1^{-1}$ to see that the complex series
$$ \left\langle w', Y_W^L\left(Y_V\left(e^{(z_1+z_0)L(1)}(-(z_1+z_0)^{-2})^{L(0)}u, -\frac{z_0z_1^{-1}}{(z_1+z_0)}\right) e^{z_1L(1)}(-z_1^{-2})^{L(0)}{u_1},z_1^{-1}\right)w\right\rangle
$$
converges absolutely when $|z_1|>|z_0|>0$ to a rational function with the only possible poles at $z_0=0, z_1=0, z_1+z_0 = 0$. And this series is precisely the complex series obtained from substituting $x_0=z_0$ and $x_1=z_1$ in the formal series (\ref{YWO-ass-2}).

We then perform the transformation $z_0\mapsto z_2-z_1$ to see that the complex series
\begin{align*}
& \left\langle w', Y_W^L\left(Y_V\left(e^{z_2L(1)}(-z_2^{-2})^{L(0)}u, -\frac{z_2-z_1}{z_2z_1}\right) e^{z_1L(1)}(-z_1^{-2})^{L(0)}{u_1},z_1^{ - 1}\right)w\right\rangle\\
= & \langle w', Y_W^L(Y_V(e^{z_2L(1)}(-z_2^{-2})^{L(0)}u, -z_1^{-1}+z_2^{-1}) e^{z_1L(1)}(-z_1^{-2})^{L(0)}{u_1},z_1^{ - 1})w\rangle
\end{align*}
which is equal to 
$$\langle w', Y_W^o(Y_V(u_2, z_2-z_1)u_1, z_1)w\rangle,$$
converges absolutely when $|z_1|>|z_1-z_2|>0$ to a rational function with the only possible poles at $z_1=0,z_2=0, z_1=z_2$. 

Now we use the definition of $Y_W^o$ to rewrite the left-hand-side as
$$\langle w',Y_W^L({e^{{z_2}L(1)}}{( - z_2^{ - 2})^{L(0)}}{u_2},z_2^{ - 1})Y_W^L({e^{{z_1}L(1)}}{( - z_1^{ - 2})^{L(0)}}{u_1},z_1^{ - 1})w\rangle$$
This series converges absolutely when $|z_1^{-1}|>|z_2^{-1}|>0$ to a rational function with the only possible poles at $z_1=0, z_2=0, z_1=z_2$. Moreover, by associativity, when $|z_2^{-1}|>|z_1^{-1}|>|z_1^{-1}-z_2^{-1}|>0$, i.e., $|z_1|>|z_2|>|z_1-z_2|>0$, it is equal to 
$$\langle w', Y_W^L({Y_V}({e^{{z_2}L(1)}}{( - z_2^{ - 2})^{L(0)}}{u_2},z_2^{ - 1} - z_1^{ - 1}){e^{{z_1}L(1)}}{( - z_1^{ - 2})^{L(0)}}{u_1},z_1^{ - 1})w\rangle $$
Thus left-hand-side is equal to right-hand-side when $|z_1|>|z_2|>|z_1-z_2|>0$. 
\end{proof}

\subsection{Contragredient of a M\"obius left $V$-module}

We first discuss the results for grading-restricted M\"obius left $V$-modules. Then we deal with the non-grading-restricted case with a stronger pole-order condition. 

\begin{thm}\label{W'Module}
Let $(V, Y_V, \one, \rho_V)$ be a M\"obius MOSVA and $(W, Y_W^L, \rho_W)$ be a grading-restricted M\"obius left $V$-module. On the graded dual $W'= \coprod_{n\in \C} W_{[n]}^*$, we define a vertex operator action of $V$ by
$$\langle Y_W'(u, z)w', w\rangle = \langle w', Y_W^o(u, z)w\rangle = \langle w', Y_W^L(e^{zL(1)}(-z^{-2})^{L(0)}u,z^{-1})w\rangle, $$
and an $\mathfrak{sl}(2)$-action $\rho_W'$ by $\rho_W'(L_j) = L'(j)$ for $j=0, \pm 1$, where 
$$\langle L_W'(j)w', w\rangle = \langle w', L_W(-j)w\rangle.$$
Then $(W', Y_W', \rho_W')$ forms a M\"obius left $V^{op}$-module. 
\end{thm}

\begin{proof} The commutator formulas for $L_W'(0), L_W'(-1)$ and $L_W'(1)$ follows from the computations in \cite{HLZ1}, Lemma 2.22. The argument there carries over to MOSVAs and requires some work. For brevity we will not include them here but redirect the reader to \cite{HLZ1}, Page 59 to 61. From Remark \ref{leftMobVopModule}, it suffice to verify that $(W', Y_W', \d_{W}', L_W'(-1))$ forms a left $V^{op}$-module. 
\begin{enumerate}
\item The lower bound condition obviously hold. The $\d$-grading condition and $\d$-commutator formula follow from the discussions in Remark \ref{d-operator}. In particular, from the $\d$-commutator formula and the lower bound condition, one sees that the series $Y_W'(u, z)w'$ is lower truncated. 
\item To see the identity property, note that $L_V(1)\one = 0$ and $L_V(0)\one = 0$, thus $e^{xL(1)}(-x^{-2})^{L(0)}\one = \one$. So $Y_W^o(\one, x) = Y_W^L(\one, x^{-1}) = 1_W$. Then follows $Y_W'(\one, x)w' = w'$. 
\item The $L(-1)$-derivative property is verified in \cite{FHL}. See \cite{FHL}, Page 47 and 48. 
\item Since $W$ is grading restricted, $(W')' = W$. Thus for the rationality of products, it suffices to verify that for every $w'\in W', w\in W, u_1, ..., u_n \in V$, the series
$$\langle Y_W'(u_1, z_1)\cdots Y_W'(u_n, z_n)w', w\rangle = \langle w', Y_W^o(u_n, z_n)\cdots Y_W^o(u_1, z_1)w\rangle$$
converges absolutely when $|z_1|>\cdots >|z_n|>0$ to a rational function with the only possible poles at $z_i=0, i=1,...,n$ and $z_i=z_j, 1\leq i < j \leq n$. This was shown in Proposition \ref{YWO-rat}. For the rationality of iterates, it suffices to show that for every $w'\in W', w\in W, u_1, u_2\in V$, the series
$$\langle Y_W'(Y_V^s(u_1, z_1-z_2)u_2, z_2)w', w\rangle$$ 
converges absolutely when $|z_2|>|z_1-z_2|>0$ to a rational function with the only possible poles at $z_1=0, z_2=0$ and $z_1=z_2$. We first use the definition of $Y_V^s$, then use $L(-1)$-conjugation property to see that
\begin{align*}
& \langle Y_W'(Y_V^s(u_1, z_1-z_2)u_2, z_2)w', w\rangle\\
= & \langle Y_W'(e^{(z_1-z_2)L(-1)}Y_V(u_2, z_2-z_1)u_1, z_2)w', w\rangle\\
= & \langle Y_W'(Y_V(u_2, z_2-z_1)u_1, z_1)w', w\rangle
\end{align*}
Note that from Remark \ref{Taylor}, this series is still in variables $z_2$ and $z_1-z_2$, where $z_1$ should be regarded as the sum $z_2+z_1-z_2$ and thus negative powers of $z_1$ should be expanded as power series in $(z_1-z_2)$. Then we use the definition of $Y_W'$ to see that this series is equal to 
$$\langle w', Y_W^o(Y_V(u_2, z_2-z_1)u_1, z_1)w\rangle.$$
And the proof of Proposition \ref{YWO-ass} shows that this series converges absolutely when $|z_2|>|z_1-z_2|>0$ to the same rational function as $\langle w', Y_W^o(u_2, z_2)Y_W^o(u_1, z_1)w\rangle$
\item The associativity follows from the discussion above and Proposition \ref{YWO-ass}. 
\end{enumerate}
\end{proof}

\begin{defn}
The module $(W', Y_W', \rho_W')$ is referred as the contragredient module of $(W, Y_W^L, \rho_W)$. In case there is no confusion, we just use $W'$ to denote it. From the results in Section 5 and Proposition \ref{Opp-Mob-MOSVA} , one easily sees that with the skew-symmetric opposite vertex operator of $Y_{W}'$, $W'$ also a M\"obius right $V$-module. 
\end{defn}

\begin{rema}
When $W$ is not grading restricted, one has to verify the rationality with $w$ taking value in the much larger space $(W')'$. So the above proof does not work. To construct the contragredient module for non-grading-restricted modules, an additional condition has to be assumed. 
\end{rema}

\begin{thm}\label{W'Module-1}
Let $(V, Y_V, \one, \rho_V)$ be a M\"obius MOSVA and $(W, Y_W^L, \rho_W)$ be a M\"obius left $V$-module. If the vertex operator $Y_W^L$ satisfies the \textit{strong pole-order condition}, that there exists a real number $C$, such that for every homogeneous $u_1, u_2\in V, w'\in W'$ and $w\in W$, the order of the pole $z_1=0$ in the rational function given by 
$$\langle w', Y_W^L(u_1, z_1)Y_W^L(u_2, z_2)w\rangle$$
is bounded above by $\wt u_1 + \text{Re}(\wt w)+ C$, then with $Y_W'$ and $\rho_W'$ are defined in the same way as the above theorem, $(W', Y_W', \rho_W')$ forms a M\"obius left $V^{op}$-module. 
\end{thm}

\begin{proof}
It suffices to deal with the rationality and associativity axioms. The idea is to use the formal variable approach. With the strong pole-order condition, we proceed to verify the weak associativity with the pole-order condition in Theorem \ref{Formal-left} based on the results of Proposition \ref{YWO-rat} and \ref{YWO-ass}. Then the conclusion follows from the theorem. 

Let $w'\in W', u_1, u_2\in V, w\in W$ be homogeneous. We rewrite the series 
$$\langle Y_W'(u_1, z_1)Y_W'(u_2, z_2)w', w\rangle = \langle w', Y_W^o(u_2, z_2)Y_W^o(u_1, z_1)w\rangle $$
as 
\begin{samepage}
\begin{align}\label{Rat-2-YWO}
\sum\limits_{{m_1} = 0}^{\text{finite}}  \sum\limits_{{m_2} = 0}^{\text{finite}}  {( - 1)}^{\text{wt }u_2 + \text{wt }u_1}\frac{1}{{{m_1}!}}\frac{1}{{{m_2}!}} & z_2^{{m_2} - 2\text{wt }u_2}z_1^{{m_1} - 2\text{wt }u_1}  \cdot \\
& \langle w',Y_W^L(L{(1)^{{m_2}}}{u_2},z_2^{ - 1})Y_W^L(L{(1)^{{m_1}}}{u_1},z_1^{ - 1})w\rangle \nonumber
\end{align}
\end{samepage}
We shall use the computations in Remark \ref{numdegree} to give an explicit upper bound of the order of the pole $z_1=0$ for the rational function given by each term in the sum. 

For each fixed $m_1, m_2$, the rational function determine by $\langle w', Y_W^L(L(1)^{m_2}u_2, z_2^{-1})Y_W^L(L(1)^{m_1}u_1, z_1^{-1})w\rangle$ is of the form
\begin{equation*}
\frac{{f(z_2^{ - 1},z_1^{ - 1})}}{{z_2^{ - {p_1}}z_1^{ - {p_2}}{{(z_2^{ - 1} - z_1^{ - 1})}^{{p_{12}}}}}}=\frac{{f(z_2^{ - 1},z_1^{ - 1})}}{{z_2^{ - {p_1} - {p_{12}}}z_1^{ - {p_2} - {p_{12}}}{{({z_1} - {z_2})}^{{p_{12}}}}}}
\end{equation*}
Note that $Y_W^L$ satisfies the strong pole-order condition, thus
$$p_1 \leq \wt (L(1)^{m_2} u_2) + \text{Re}(\wt w) + C$$
Let $d$ be the degree of $f$ as a polynomial in $z_2^{-1}, z_1^{-1}$. From Remark \ref{numdegree},  
$$d = p_1 + p_2 + p_{12} + \wt w' - \wt (L(1)^{m_2}u_2) - \wt (L(1)^{m_1} u_1) - \wt w$$
Note that though $\wt w'$ and $\wt w$ might be complex numbers, their difference is supposed to be an integer. In particular, we know that 
$$\wt w' - \wt w = \text{Re}(\wt w') - \text{Re}(\wt w)$$
If we write 
$$f(x_1, x_2) = \sum_{k=0}^d a_k x_1^k x_2^{d-k},$$
then 
$$f(z_2^{-1}, z_1^{-1}) = \sum_{k=0}^d a_k z_2^{-k} z_1^{-d+k}$$
where the lowest possible power of $z_1$ is $-d$. Therefore, the order of pole of the rational function that each term in (\ref{Rat-2-YWO}) converges to is bounded above by
\begin{align*}
d - p_2 - p_{12} - m_1 + 2 \wt u_1 & = p_1 +\wt w' - \wt (L(1)^{m_2}u_2) - \wt (L(1)^{m_1} u_1) - \wt w - m_1 + 2 \wt u_1\\
& =  p_1 + \text{Re}(\wt w') - \wt (L(1)^{m_2}u_2)  -\text{Re}(\wt w) + \wt u_1\\
& \leq  \text{Re}(\wt w') + \wt u_1 + C.
\end{align*}
This upper bound is independent of $m_1, m_2$. Thus we have proved that the order of the pole $z_1=0$ of the rational function given by 
$$\langle Y_W'(u_1, z_1)Y_W'(u_2, z_2)w', w\rangle = \langle w', Y_W^o(u_2, z_2)Y_W^o(u_1, z_1)w\rangle $$
is controlled above by the real number that depends only $u_1$ and $w_1$. So with the assumption here, the vertex operator $Y_W'$ satisfies the pole-order condition as in Definition \ref{NSHTCL}. 

Now with the conclusion of Proposition \ref{YWO-rat} and \ref{YWO-ass}, we know that one can choose $q_1 = \wt w' + \wt u_1 + C$ depending only on $u_1$ and $w'$, $q_2$ depending only on $u_2$ and $w'$, $q_{12}$ depending only on $u_1$ and $u_2$, such that 
$$(z_0+z_2)^{q_1}z_2^{q_2}z_0^{q_{12}}\langle Y_W'(u_1, z_0+z_2)Y_W'(u_2, z_2)w', w\rangle = (z_0+z_2)^{q_1}z_2^{q_2}z_0^{q_{12}}\langle Y_W'(Y_V^s(u_1, z_0)u_2, z_2)w', w\rangle$$
converges absolutely to a polynomial function. 
Thus as formal series with coefficients in $W'$,
$$(x_0+x_2)^{q_1}x_2^{q_2}x_0^{q_{12}}\langle Y_W'(u_1, x_0+x_2)Y_W'(u_2, x_2)w' = (x_0+x_2)^{q_1}x_2^{q_2}x_0^{q_{12}}\langle Y_W'(Y_V^s(u_1, x_0)u_2, x_2)w'\in W'[[x_0, x_2]].$$
has no negative powers of $x_0$ and $x_2$. Thus they all live in $W'[[x_0, x_2]]$. The weak associativity relation is then seen by dividing both sides by $x_2^{q_2}$ and $x_0^{q_{12}}$. Moreover, the choice of $q_1$ depends only on $u_1$ and $w'$. Thus, as a consequence of Theorem \ref{Formal-left}, the rationality and associativity axioms hold. 
\end{proof}

\begin{rema}
This strong pole-order condition is natural because it is satisfied by all the M\"obius left modules for vertex algebras. 
\end{rema}

\noindent {\small \sc Department of Mathematics, Rutgers University, Piscataway, NJ  08854-8019, USA}

\noindent {\em E-mail address}: fq15@scarletmail.rutgers.edu, 
\end{document}